%% file: Splitting_LowMach.tex
\documentclass[10pt]{article}
\usepackage{amssymb,amsmath}
\usepackage{amsthm}
\usepackage{graphicx}
\usepackage{eucal}
\newcommand{\R}{{\mathbb R}}

\newcommand{\BO}{{\rm C}_{\rm b}(\R)}
\newcommand{\BOinf}{{\rm C}^{\infty}_{\rm b}(\R)}
\newcommand{\Cinf}{{\rm C}^{\infty}}

\def\ds{\displaystyle} 

\newcommand{\lie}{{\CMcal L}}
\newcommand{\strang}{{\CMcal S}}
\newcommand{\Or}{{\CMcal O}}

\newcommand{\Id}{{\rm Id}}
\newcommand{\der}{{\mathrm{d}}}
\newcommand{\err}{{\mathrm{E}}}
\newcommand{\e}{{\mathrm e}}
\newcommand{\maxu}{{\kappa}}
\newcommand{\Deru}{{\mathcal{D}}}

\def\split{\Delta t}

\newtheorem{theorem}{Theorem}
\newtheorem{corollary}{Corollary}

\title{Analysis of operator splitting in the non-asymptotic regime 
for nonlinear reaction-diffusion equations.
Application to the dynamics of 
premixed flames}

\author{
St\'ephane~Descombes\footnotemark[2]\ \footnotemark[7]
\and
Max~Duarte\footnotemark[2]\ \footnotemark[6]
\and
Thierry~Dumont\footnotemark[3]
\and
Fr\'ed\'erique~Laurent\footnotemark[4]
\and
Violaine~Louvet\footnotemark[3]
\and
Marc~Massot\footnotemark[4]
}

\begin{document}

\maketitle

\renewcommand{\thefootnote}{\fnsymbol{footnote}}

\footnotetext[2]{
Univ. Nice Sophia Antipolis, CNRS,  LJAD, UMR 7351, 06100 Nice, France.
}

\footnotetext[7]{
INRIA Sophia Antipolis-Mediterran\'e Research Center, Nachos project-team,
06902 Sophia Antipolis Cedex, France
({\tt stephane.descombes@unice.fr}).
}

\footnotetext[6]{
CCSE,
Lawrence Berkeley National Laboratory,
1 Cyclotron Rd. MS 50A-1148,
Berkeley, CA 94720, USA
({\tt MDGonzalez@lbl.gov}).
}

\footnotetext[3]{
Institut Camille Jordan - UMR CNRS 5208,
Universit\'e de Lyon,
Universit\'e Lyon 1,
INSA de Lyon 69621, 
Ecole Centrale de Lyon,
43 Boulevard du 11 novembre 1918,
69622 Villeurbanne Cedex, France
({\tt \{tdumont,louvet\}@math.univ-lyon1.fr}).
}

\footnotetext[4]{
Laboratoire EM2C - UPR CNRS 288, Ecole Centrale
 Paris, 
Grande Voie des Vignes, 92295 Ch\^{a}tenay-Malabry Cedex,
 France
({\tt \{frederique.laurent,marc.massot\}@ecp.fr}).
}

\renewcommand{\thefootnote}{\arabic{footnote}}

\begin{abstract}
In this paper we mathematically characterize 
through a Lie formalism the 
local errors induced by 
operator splitting
when solving nonlinear
reaction-diffusion equations, especially in the non-asymptotic regime. 
The non-asymptotic regime is often attained in practice when the splitting time step is 
much larger than some of the scales associated with either source terms or the diffusion operator 
when large gradients are present. 
In a series of previous works
a reduction of the asymptotic orders
for a range of large splitting time steps related to 
very short time scales in the nonlinear 
source term has been studied, as well as that 
associated with large gradients but for linearized equations.
This study provides a key theoretical step forward since it 
characterizes the numerical behavior of splitting errors
within a more general nonlinear framework,
for which new error estimates can be derived
by coupling Lie formalism
and regularizing effects of the heat equation. 
The validity of these theoretical results is
then assessed in the framework of 
two numerical applications,
a KPP-type reaction wave where 
the influence of stiffness on local error estimates can be
thoroughly investigated;
and a much more complex problem,
related to premixed flame dynamics in the low Mach number regime with 
complex chemistry and detailed transport,
for which the present theoretical study 
shows to provide relevant insights.
\end{abstract}

\subsubsection*{Keywords} 
Operator splitting, error bounds, reaction-diffusion, traveling waves, combustion

\subsubsection*{AMS subject classifications}
65M15, 65L04, 65Z05, 35A35, 35K57, 35C07

\pagestyle{myheadings}
\thispagestyle{plain}
\markboth{DESCOMBES, DUARTE, DUMONT, LAURENT, LOUVET, MASSOT}
{SPLITTING ERRORS FOR NONLINEAR REACTION-DIFFUSION EQUATIONS}

\section{Introduction}
Operator splitting techniques \cite{Strang68,Marchuk90}, 
also called fractional steps methods
\cite{Temam1,Temam2,Yanenko71}, were introduced
with the main objective of saving
computational costs.
A complex and potentially large problem
could be then split into smaller or subproblems of different nature 
with an important reduction
of the algorithmic complexity and computational requirements.
The latter characteristics were largely exploited over the past years
to carry out numerical simulations in several domains
going from biomedical models,
to combustion or air pollution modeling applications.
Moreover,
these methods continue to be widely used mainly because
of their simplicity of implementation
and their high degree of liberty 
in terms of choice of dedicated numerical solvers for
the split subproblems.
They are in particular suitable 
for stiff problems,
for which a special care must be addressed to choose 
adequate and stable methods that properly handle and damp out
fast transients inherent,
for instance, to the reaction
\cite{Verwer1999}
or diffusion \cite{Ropp2005} equations.
In most applications first and second order splitting schemes
are implemented for which a general mathematical background
is available (see, {\it e.g.}, \cite{MR2221614} for ODEs and
\cite{MR2002152} for PDEs).
Even though these schemes are usually efficient for the solution of
time dependent equations,
it is well known that they might
suffer from 
order reduction
in the stiff case, 
and some studies were conducted 
to explain this phenomenon.
Another potential issue is the accuracy loss
related to the boundary conditions 
for PDEs with transport operators, solved
in a splitting framework.
This problem was investigated, for instance, in
\cite{Hundsdorfer95,MR2002152}
for advection-reaction equations and
mathematically described in \cite{MR2545819}
in a more general framework
for two linear operators generating strongly continuous semigroups.
For stiff applications,
several works \cite{Dangelo95,Yang1998,Sportisse2000,Ropp04}
illustrated perturbing effects on the accuracy of splitting
approximations for multi-scale PDEs.
Multi-scale features in time are commonly related
to physical dynamics characterized by 
a broad range of time scales, 
while steep gradients or large higher order spatial derivatives
induce similar phenomena in space.
In all these cases
the standard numerical analysis of
splitting errors remains valid for asymptotically small
time steps, and rapidly becomes insufficient for stiff problems.
A better understanding of splitting methods
for such regimes can be thus
justified by the fact that
practical considerations often suggest the use
of relatively large time steps in order to 
ease heavy computational costs related to 
the numerical simulation of complex applications.

For PDEs disclosing physical time scales much faster than the
splitting time step,
a theoretical study 
was conducted in \cite{Sport2000} in the framework of a linear system of
ODEs issued from a reaction-diffusion equation
with a linear source term and diagonal diffusion.
Splitting errors with relatively large splitting time steps
were therefore mathematically described,
whereas splitting schemes ending with the stiffest operator 
were also shown to be more accurate.
Similar conclusions were drawn in
\cite{Kozlov2004} for nonlinear systems of ODEs.
A mathematical framework 
was then introduced in \cite{DesMas04}
to describe these errors 
for nonlinear reaction-diffusion equations.
This work further analyzed 
order reduction 
in direct relation to the nonlinearity
of the equations, and in particular
confirmed better performances for splitting schemes
that finish with the time integration of the stiffest
operator.
Other theoretical studies were also conducted
to investigate splitting
errors and in particular to derive alternative estimates
exhibiting 
deviations from classical asymptotic estimates.
A numerical analysis based on analytic semi-group theory was first
considered in \cite{MR1799313} for linear operators, and then
in \cite{Descombes07} 
for a system of ODEs issued from a discretized linear 
reaction-diffusion equation 
with solutions of high spatial gradients.
The latter approach, 
based on the exact representation of local splitting errors introduced
in \cite{Descombes02},
was then recast in \cite{Duarte11_parareal} 
in infinite dimension for a linearized reaction-diffusion equation.
Whether the analysis is performed in finite or infinite dimension,
the resulting estimates predict an effective 
order reduction 
for 
linear or linearized reaction-diffusion equations.
For instance, 
local errors related to
a Lie approximation
of first order 
exhibits 
deviations from $\Or(\split^2)$ observed in the asymptotic regime to $\Or(\split^{1.5})$
for a range of relatively large splitting time step $\split$
\cite{Descombes07,Duarte11_parareal}.
Similarly, Strang error approximations deviate
 from 
$\Or(\split^3)$ to $\Or(\split^{2})$ 
in infinite dimension \cite{Duarte11_parareal},
or from 
$\Or(\split^3)$ in the asymptotic regime to $\Or(\split^{2.5})$,
and potentially $\Or(\split^{1.5})$
in various ranges of splitting time steps
for the corresponding semi-discretized
problem  \cite{Descombes07}.

All of these studies shed some light on the behavior of splitting
methods for stiff PDEs and in particular for non-arbitrarily small
splitting time steps.
Nevertheless, a mathematical description
in a more general and fully nonlinear framework
seems natural to further
investigate these schemes.
No rigorous analysis of these configurations
is however available so far in the literature.
The relevance of such a study is hence justified not only because
most of physical models disclose important stiffness 
but because short splitting time steps heavily restrict the efficiency
of splitting methods.
A better understanding of these schemes for non-asymptotic regimes,
that is, for splitting time steps much larger than the 
fast scales associated with each operator,
seems therefore necessary
to enhance the numerical performance of such methods.

We conduct in this study the numerical analysis
of splitting errors for time dependent PDEs in the case of nonlinear
reaction-diffusion equations.
The approach adopted is based on previous analyses carried out 
with linear operators and analytic semi-group theory, 
as well as the exact representation of splitting errors.
The inherent nonlinearity of the equations is handled through
the Lie formalism.
In this work we limit the study to diagonal diffusion terms
as a first step, and we neglect as well the influence of boundary
conditions of the PDEs.
We derive local error bounds that consistently
describe classical orders, as well as 
a hierarchy of 
alternative estimates more relevant
in  non-asymptotic regimes related to various ranges of large splitting time step.
In particular for large splitting times
and problems modeling steep fronts,
such a mathematical characterization 
shows that 
this deviation from the asymptotic behavior actually involves
smaller numerical
errors than the ones expected with the asymptotic 
classical
orders.
The resulting theoretical estimates are then evaluated for PDEs
modeling traveling waves, for which stiffness can be easily introduced
in the equations and thus allows us to systematically investigate
various stiff scenarios.
To further assess these theoretical findings for more complex and
realistic applications,
we investigate splitting errors
for the simulation of premixed flame dynamics in the low Mach number regime with 
complex chemistry and detailed transport.
We therefore introduce a new splitting method compatible with the low Mach number constraint 
and show how 
the theoretical results we have obtained allow us to gain fundamental insight in the analysis of 
splitting errors, thus paving the way 
for further theoretical 
studies as well as new numerical algorithms.

The paper is organized as follows.
We carry out the numerical analysis of operator splitting in Section
\ref{sec:theory}, for nonlinear reaction-diffusion equations.
We then evaluate in Section \ref{sec:KPP} the previous 
theoretical estimates in the context of
PDEs modeling traveling waves, in particular with a KPP-type of nonlinearity. 
A counterflow premixed flame is studied in Section \ref{sec:complex},
in the low Mach number regime
with complex chemistry and detailed transport.

\section{Analysis of operator splitting errors in the non-asymptotic regime}\label{sec:theory}
In this section we conduct a mathematical description
of splitting local errors for nonlinear reaction-diffusion equations.
First, we recall some previous theoretical results
for operator splitting in a linear framework,
and then we 
investigate the nonlinear case by 
using Lie derivative calculus.

\subsection{Error formulae in the linear framework}
Let us consider two general linear operators $A$ and $B$, for which the exponentials 
$\e ^{-tA}$ and $\e ^{-tB}$ can be understood as a formal series.
The first order Lie and the second order Strang splitting formulas 
to approximate $\e ^{-t(A+B)}$
are,
respectively, given by
\begin{equation}\label{lie_strang_linear}
\lie(t) =\e^{-t A}\e^{-t B}, \quad
\strang(t) = \e^{-t A/2} \e^{-t B} \e^{-t A/2} .
\end{equation}
In what follows
we will give an exact representation of the difference between 
$\e ^{-t(A+B)}$ and its Lie and Strang approximations (\ref{lie_strang_linear}),
by recalling some results proved in \cite{Descombes07} and \cite{Descombes02}.
We introduce the following notations: 
$\partial_A B$ denotes the commutator between $A$ and $B$,
\begin{equation}\label{eq1:commutator_AB}
\partial_A B = [A,B]=AB-BA,
\end{equation}
and thus
\begin{equation}\label{eq1:commutator_AB_strang}
\partial^2_A B = \Big[A,[A,B]\Big], \qquad 
\partial^2_B A = \Big[B,[B,A]\Big].
\end{equation}
\begin{theorem}
The following
identities hold
\begin{align}
\lie(t) = &\ \e^{-t(A+B)} 
+\int_0^t \int_0^s \e^{-(t-s)(A+B)} \e^{-(s-r)A} \bigl(\partial_A B\bigr) \e^{-rA}\e^{-sB}
\, \der r\, \der s,   \label{eq1:lie_exact_error} \\[0.5ex]
\strang(t)  = &\ \e^{-t(A+B)}  \nonumber \\[0.5ex]
 & + \frac{1}{4} \int_0^t \int_0^s
(s-r)\e^{-(t-s)(A+B)} \e^{-(s-r)A/2} \bigl(\partial^2_A B\bigr) \e^{-rA/2}\e^{-sB}\e^{-sA/2}
\, \der r\, \der s  \nonumber \\[0.5ex]
& - \frac{1}{2}\int_0^t \int_0^s
(s-r)\e^{-(t-s)(A+B)} \e^{-sA/2}\e^{-rB} \bigl(\partial^2_B A\bigr) \e^{-(s-r)B}\e^{-sA/2}
\, \der r\, \der s. \label{eq1:strang_exact_error}
\end{align}
\end{theorem}
Formula \eqref{eq1:lie_exact_error} was originally introduced in \cite{ShengQin}.
Additionally, we have the following equivalent representations
which turn out to be more convenient in the nonlinear case.
\begin{corollary}\label{col:linear_err}
Considering \eqref{eq1:lie_exact_error} and \eqref{eq1:strang_exact_error},
the following
identities hold
\begin{align}
\lie(t)  = & \ \e^{-t(A+B)} 
+\int_0^t \int_0^s 
\e^{-sA}\e^{-rB}
\bigl(\partial_A B\bigr) 
 \e^{-(s-r)B}
\e^{-(t-s)(A+B)} 
\, \der r\, \der s,   \label{eq1:lie_exact_error2} \\[0.5ex]
\strang(t)  = & \ \e^{-t(A+B)}  \nonumber \\[0.5ex]
 & + \frac{1}{4} \int_0^t \int_0^s
(s-r)
\e^{-sA/2}
\e^{-sB}
\e^{-rA/2}
\bigl(\partial^2_A B\bigr)
\e^{-(s-r)A/2}
\e^{-(t-s)(A+B)}   
\, \der r\, \der s  \nonumber \\[0.5ex]
 & - \frac{1}{2}\int_0^t \int_0^s
(s-r)
\e^{-sA/2}
\e^{-(s-r)B}
\bigl(\partial^2_B A\bigr)
\e^{-rB}\e^{-sA/2}
\e^{-(t-s)(A+B)}   
\, \der r\, \der s. \label{eq1:strang_exact_error2}
\end{align}
\end{corollary}
\begin{proof}
It suffices to compute the adjoint of 
\eqref{eq1:lie_exact_error} and \eqref{eq1:strang_exact_error},
and noticing that according to the definition of exponentials,
we have $\left(\e^{tA}\right)^*= \e^{tA^*}$,
$\left(\e^{tA}\e^{tB}\right)^*= \e^{tB^*}\e^{tA^*}$,
and with \eqref{eq1:commutator_AB} and \eqref{eq1:commutator_AB_strang}:
$\left(\partial_A B\right)^*= \partial_{B^*} A^*$,
$\left(\partial^2_A B\right)^*= \partial^2_{A^*} {B^*}$,
and
$\left(\partial^2_B A\right)^*= \partial^2_{B^*} {A^*}$.
\end{proof}

\subsection{Splitting errors for nonlinear reaction-diffusion equations}\label{sec:err_RD}
We  consider the scalar reaction-diffusion equation
\begin{equation}\label{reacdiff}
\left.
\begin{array}{ll}
\partial_t u- \partial^2_x u =  f(u),
& x \in \R, \, t > 0,\\[0.5ex]
u(x,0)=u_0(x), & x \in \R. 
\end{array}
\right\}
\end{equation}   
Considering the maximum norm $\| \, \,\|_{\infty}$,
we denote by 
$\BO$ the space 
of functions bounded over $\R$, 
and by $\BOinf$ the functions of class $\Cinf$ bounded over $\R$.
We assume that $f$ is a function of class $\Cinf$, 
from $\R$ to itself,
such that there exists $R>0$ for which
\begin{equation}\label{assumptiononf}
|r| \geq R
\,
\Rightarrow
\,
r f(r) \leq 0.
\end{equation}
Without loss of generality we assume that $f(0)=0$ and $R=1$.
If $u_0$ belongs to $\BOinf$, it can be then shown 
that equation \eqref{reacdiff} has a unique solution \cite{MR1672406},
and we represent the solution $u(t,.)$ as $T^t u_0$, where $T^t$ is the semi-flow
associated with \eqref{reacdiff}. 
Moreover
such a function $u$  
is infinitely differentiable over
$\R \times (0,\infty)$,
and the following estimate holds \cite{MR1672406},
\begin{equation}\label{majo}
\left \|   T^{t} u_0    \right \|_{\infty} \leq\max \left ( \|u_0\|_{\infty}, 1 \right ).
\end{equation}

Given
$v_0$  and $w_0$ in $\BOinf$, we consider the following
equations:
\begin{equation}\label{diff}
\left.
\begin{array}{ll}
\partial_t v- \partial^2_x v =  0,
& x \in \R, \, t > 0,\\[0.5ex]
v(x,0)=v_0(x), & x \in \R, 
\end{array}
\right\}
\end{equation}  
and
\begin{equation}\label{reac}
\left.
\begin{array}{ll}
\partial_t w = f(w),
& x \in \R, \, t > 0,\\[0.5ex]
w(x,0)=w_0(x), & x \in \R. 
\end{array}
\right\}
\end{equation}  
We denote by $X^t v_0$ and $Y^t w_0$, respectively, the solutions of \eqref{diff} and
\eqref{reac}. 
It is well known that for  $t \geq 0$ and $u_0$ in $\BOinf$, $\| X^t u_0\|_{\infty} \leq \| u_0 \|_{\infty}$;
furthermore property \eqref{majo} holds naturally for $Y^t$ with the assumption \eqref{assumptiononf}.
The Lie approximation formulas
are defined by
\begin{equation}\label{Lie1}
\lie_1^t u_0 =X^t Y^t u_0,
\quad
\lie_2^t u_0 = Y^t X^t u_0,
\end{equation}
whereas the two Strang approximation formulas \cite{Strang68}
are given by
\begin{equation}\label{Strang}
\strang_1^t u_0=
X^{t/2} Y^t X^{t/2} u_0,
\quad
\strang_2^t u_0=
Y^{t/2} X^t Y^{t/2} u_0.
\end{equation}
In what follows we investigate the error between the exact solution 
of equation (\ref{reacdiff}), and the corresponding
Lie approximations \eqref{Lie1}.
Results for Strang local errors can be found in Appendix \ref{Strang_local_error}.
The Strang splitting approximation error
for a semi-linear parabolic equation like (\ref{reacdiff}) was also formally characterized
in \cite{HansKramOst}.
A different approach is adopted in the present study, where
a more compact form of the representation of the error 
is considered
to investigate its dependence on the initial condition and its derivatives.
To perform these computations, we use formulas \eqref{eq1:lie_exact_error2}
and \eqref{eq1:strang_exact_error2} from Corollary \ref{col:linear_err},
and the Lie derivative calculus (see, for example, \cite{MR2221614} Sect. III.5 or
\cite{MR2002152}  Sect. IV.1.4 for an introduction to this topic).
Lie calculus was also considered  
in \cite{MR2429878,DesThal} and in \cite{MR2847238}
to study splitting schemes for, respectively, 
nonlinear Schr\"odinger equations and  nonlinear reaction-diffusion equations. 
Notice that by considering equation \eqref{reacdiff}
over $\R$, we exclude the boundary conditions from the present
theoretical study. One must recall, however, that both Dirichlet
and Neumann boundary conditions may have a negative influence
on the splitting approximations, as previously mentioned
in the Introduction (see, e.g., \cite{Hundsdorfer95,MR2002152,MR2545819}).
In particular, a recent study \cite{FaouOst2014}
mathematically investigates
this problem for both Lie and Strang approximations
applied to a two-dimensional inhomogeneous parabolic equation
(similar to \eqref{reacdiff}, but with $f(x,y,t)$,
$(x,y)\in \R^2$,
instead of $f(u)$).

Let us briefly recall in the following the definition and some properties of Lie derivatives.
We consider function $f$
as 
an unbounded nonlinear operator
in ${\BOinf}$.
For any 
unbounded nonlinear operator
$G$ 
in $\BOinf$
with Fr\'echet derivative $G'$,
the corresponding Lie derivative
$D_f$ maps $G$ to a new operator 
$D_fG$ such that 
for any $u_0$ in $\BOinf$:
\begin{equation*}
\left ( D_f G \right ) (u_0) =  \partial_t  G(Y^t  u_0) \vert_{t=0}  = G'(u_0) f(u_0).
\end{equation*}
Hence
\begin{equation*}
\left ( \e^{t D_f } G \right ) (u_0) =  G(Y^t u_0),
\end{equation*}
and for $G=\Id$, we have the following representation
of the flow of \eqref{reac}:
\begin{equation*}
\left ( \e^{t D_f } \Id  \right ) (u_0) =  Y^t u_0.
\end{equation*}
Similarly, we can write the flow associated with \eqref{diff}
by considering the corresponding Lie derivative $D_\Delta$.
We finally recall
that the commutator of Lie derivatives of two 
unbounded nonlinear operators
is the Lie derivative of the Lie bracket 
of the 
unbounded nonlinear operators
in reversed
order. 
For instance, 
the Lie bracket for $\Delta$ and $f$ is defined for any $u_0$
in $\BOinf$
by
\begin{align}\label{com_f_Delta}
\lbrace \Delta , f \rbrace (u_0) &=
\partial^2_x  ( f(u_0) )  - f'(u_0) \partial^2_x u_0
 \nonumber \\
&= 
f''(u_0) ( \partial_x u_0,\partial_x u_0 ) =
f''(u_0) ( \partial_x u_0 )^2,
\end{align}
and thus we have
\begin{equation*}
 \left( [D_f,D_\Delta]  \Id \right) (u_0)= 
\left( D_{\lbrace \Delta , f \rbrace} \Id \right) (u_0) =
\lbrace \Delta , f \rbrace (u_0).
\end{equation*}

Considering now the Lie splitting approximations \eqref{Lie1} together with
Lie derivative calculus, we have
\begin{equation}\label{lie1_derlie}
T^tu_0 - \lie_1^t u_0  =
\left ( \e^{t(D_{\Delta}+D_{f})} \Id \right ) (u_0) - \left ( \e^{tD_{f}}\e^{tD_{\Delta}}\Id \right ) ( u_0),
\end{equation}
and
\begin{equation}\label{lie2_derlie}
T^tu_0 - \lie_2^t u_0  =
\left ( \e^{t(D_{\Delta}+D_{f})} \Id \right ) (u_0) - \left ( \e^{tD_{\Delta}}\e^{tD_{f}}\Id \right ) ( u_0),
\end{equation}
which yield the following exact representations of the local error, 
denoted as
\begin{equation}\label{lie_localerr}
\err^t_{\lie_1} u_0 = T^tu_0 - X^t Y^t u_0,
\quad
\err^t_{\lie_2} u_0 = T^tu_0 - Y^t X^t u_0;
\end{equation}
in which $\Deru$ denotes the derivative with respect to the initial condition
since $T^t$, $X^t$, and $Y^t$ have been defined as semi flows.
\begin{theorem}\label{theo_exact_err}
For $t \geq 0$ and $u_0$ in $\BOinf$, we have
\begin{align}
T^tu_0 - X^t Y^t u_0 = & - \int_0^t \int_0^s
\Deru T^{t-s} (X^s Y^s u_0) X^{s-r} 
 f''(X^r Y^s u_0) \left ( \partial_x  X^r Y^s u_0 \right )^2 
\, \der r\, \der s, \label{lie1_exact_teo}
\end{align}
and
\begin{align}
T^tu_0 - Y^t X^t u_0 = &  \int_0^t \int_0^s
\Deru T^{t-s} (Y^s X^s u_0) 
\exp\left (\int_0^{s-r} f'(Y^{\sigma+r}  X^s u_0  )\, \der \sigma \right )
\, \times \nonumber \\[0.5ex]
& f''(Y^r X^s u_0 ) 
\exp\left (2\int_0^r f'(Y^{\sigma}  X^s u_0  )\, \der \sigma \right )
\left ( \partial_x  X^s u_0 \right )^2 
\, \der r\, \der s. \label{lie2_exact_teo}
\end{align}
\end{theorem}
\begin{proof}
Considering \eqref{eq1:lie_exact_error2} we have
\begin{align*}
\err^t_{\lie_1} u_0= &
\left ( \e^{t(D_{\Delta}+D_{f})} \Id \right )  (u_0) - 
\left (\e^{tD_{f}}\e^{tD_{\Delta}}\Id \right ) ( u_0)\nonumber \\[0.5ex]
= &
- \int_0^t \int_0^s
\left ( 
\e^{sD_{f}}\e^{rD_{\Delta}} [D_f,D_\Delta]\e^{(s-r)D_{\Delta}}
\e^{(t-s)(D_{\Delta}+D_{f})} \Id \right ) (u_0) 
\, \der r\, \der s\nonumber \\[0.5ex]
= &
- \int_0^t \int_0^s
\left ( 
\e^{rD_{\Delta}} [D_f,D_\Delta]\e^{(s-r)D_{\Delta}}
\e^{(t-s)(D_{\Delta}+D_{f})} \Id \right ) (Y^s u_0) 
\, \der r\, \der s\nonumber \\[0.5ex]
= &
- \int_0^t \int_0^s
\left ( 
D_{\lbrace \Delta , f \rbrace}
 \e^{(s-r)D_{\Delta}}
\e^{(t-s)(D_{\Delta}+D_{f})} \Id \right ) (X^r Y^s u_0) 
\, \der r\, \der s\nonumber \\[0.5ex]
= &
- \int_0^t \int_0^s
\left (
\e^{(s-r)D_{\Delta}}
\e^{(t-s)(D_{\Delta}+D_{f})} \Id \right )'(X^r Y^s u_0) 
\lbrace \Delta , f \rbrace (X^r Y^s u_0) 
\, \der r\, \der s. 
\end{align*} 
Since $\left (
\e^{(s-r)D_{\Delta}}
\e^{(t-s)(D_{\Delta}+D_{f})} \Id \right ) (v_0) = T^{t-s}X^{s-r} v_0$,
we obtain that
\begin{align}
\err^t_{\lie_1} u_0   = 
- \int_0^t \int_0^s
\Deru T^{t-s} (X^s Y^s u_0)  \Deru X^{s-r} (X^r Y^s u_0)
\lbrace \Delta , f \rbrace (X^r Y^s u_0)
\, \der r\, \der s. \label{lie1_exact_general_teo}
\end{align}
Formula  \eqref{lie1_exact_teo} follows from \eqref{lie1_exact_general_teo} as
a consequence of \eqref{com_f_Delta}.
By formally exchanging $f$ and $\Delta$ (and thus $Y$ and $X$)
we infer from  \eqref{lie2_derlie} and  \eqref{eq1:lie_exact_error2} that
\begin{align}
\err^t_{\lie_2} u_0  = 
 - \int_0^t \int_0^s
\Deru T^{t-s} ( Y^sX^s u_0)  \Deru Y^{s-r} (Y^r X^s u_0)
\lbrace f,\Delta \rbrace (Y^r X^s u_0)
\, \der r\, \der s. \label{lie2_exact_general_teo}
\end{align}
Given $w_0$, the derivative of the solution of 
\eqref{reac}, $Y^s w_0(x)$, with respect to the 
initial condition, denoted by $\Deru Y^s (w_0)$,
satisfies
\begin{equation*}\label{Dreac}
\left.
\begin{array}{l}
\partial_s \Deru Y^s (w_0) = f'(Y^s w_0)\Deru Y^s (w_0),
\\[0.5ex]
\Deru Y^0 (w_0)= 1, 
\end{array}
\right\}
\end{equation*}
and since $f$ is a scalar function 
\begin{equation}\label{derivdure}
\Deru Y^s (w_0)=\exp \left (\int_0^s f'(Y^{\sigma}w_0) d\sigma \right ).
\end{equation}
Similarly, $\partial_x Y^s w_0$ satisfies
\begin{equation*}\label{Dx_reac}
\left.
\begin{array}{l}
\partial_s \partial_x Y^s w_0 = f'(Y^s w_0) \partial_x Y^s w_0,
\\[0.5ex]
\partial_x Y^0 w_0= \partial_x w_0, 
\end{array}
\right\}
\end{equation*}
and hence
\begin{equation}\label{derx_ivdure}
\partial_x Y^s w_0=\exp \left (\int_0^s f'(Y^{\sigma}w_0) d\sigma \right ) \partial_x w_0,
\end{equation}
which along with \eqref{com_f_Delta} and \eqref{derivdure} into \eqref{lie2_exact_general_teo},
prove \eqref{lie2_exact_teo}.
\end{proof}

Notice that both error representations \eqref{eq1:lie_exact_error}
and \eqref{eq1:lie_exact_error2} are equivalent, nevertheless
we will see in the following that the second one,
used in Theorem \ref{theo_exact_err}, is more convenient
because operators of type $\partial_x$ are applied only
on the split solutions 
$X^t v_0$ and $Y^t w_0$ of \eqref{diff} and
\eqref{reac}.
A rigorous proof for \eqref{lie1_exact_general_teo} and \eqref{lie2_exact_general_teo}
for two general nonlinear operators was also proposed in \cite{DesThal}.
Furthermore,
using Duhamel's formula we have for equation \eqref{reacdiff}: 
\begin{align*}
T^s u_0 = X^s u_0 + \int_0^s X^{s-r} f(T^r u_0) \, \der r,
\end{align*}
and hence
\begin{align}\label{duhamel_coupled_DT}
\Deru T^s (u_0) = X^s  + \int_0^s X^{s-r} f'(T^r u_0) \Deru T^r(u_0) \, \der r.
\end{align}
In particular
using Gronwall's lemma
we can also demonstrate that
\begin{align*}
\Deru T^s (u_0) = X^s \left( 1 + \int_0^s f'(T^r u_0) 
\exp\left (\int_r^s X^{s-\sigma} f'(T^{\sigma} u_0  )\, \der \sigma \right )
 \, \der r \right),
\end{align*}
and hence have explicit expressions for both 
\eqref{lie1_exact_teo} and \eqref{lie2_exact_teo}.
We introduce now the following notation: for a scalar continuous function $g$ 
and a real $a$, we denote
\begin{equation*}
\| g \|_{[-a,a]} = \sup_{r \in [-a,a]} |g(r)|.
\end{equation*}
Using the results of Theorem \ref{theo_exact_err},
the following bounds can be obtained for 
both Lie local errors 
\eqref{lie1_derlie} and \eqref{lie2_derlie}.
\begin{theorem}\label{theo_order_class}
For $t \in [0,{\rm T})$  and $u_0$ in $\BOinf$,
with $\maxu = \max ( \| u_0 \|_{\infty},1)$,
we have
\begin{equation}\label{lie1_order_2}
\left \| 
T^tu_0 -  X^t Y^t u_0
\right \|_{\infty} 
\leq
\frac{t^2 \exp \left ( 2t {\|f'\|_{[-\maxu,\maxu]}} \right )
\|f''\|_{[-\maxu,\maxu]} }{2}
\left \| 
\partial_x u_0 
\right \|^2_{\infty},
\end{equation}
and
\begin{equation}\label{lie2_order_2}
\left \| 
T^tu_0 - Y^t X^t  u_0
\right \|_{\infty} \leq
\frac{t^2 \exp \left ( 2t {\|f'\|_{[-\maxu,\maxu]}} \right )
\|f''\|_{[-\maxu,\maxu]} }{2}
\left \| 
\partial_x u_0   
\right \|^2_{\infty}.
\end{equation}
\end{theorem}
\begin{proof}
Taking norms for \eqref{lie1_exact_general_teo}, we obtain
\begin{align*}
 \left \| T^tu_0 - X^t Y^t  u_0  \right \|_{\infty} \leq  & 
\int_0^t \int_0^s
\left \|  \Deru T^{t-s} (X^s Y^s u_0) \right \|_{\infty} 
\left\|  f''(X^r Y^s u_0 ) \right \|_{\infty} 
\, \times \nonumber \\[0.5ex]
& \left \| \left (\partial_x  X^r Y^s u_0  \right )^2 \right \|_{\infty} 
\, \der r\, \der s. 
\end{align*} 
From \eqref{duhamel_coupled_DT}, assumption \eqref{assumptiononf},
and hence \eqref{majo},
we see that
\begin{align*}
\left \| \Deru T^s (u_0) \right \|_{\infty} & \leq 
1  + \int_0^s \left \| f'(T^r u_0) \Deru T^r(u_0) \right \|_{\infty} \, \der r 
\nonumber \\[0.5ex]
& \leq 
1  + \int_0^s \| f' \|_{[-\maxu,\maxu]} 
\left \|  \Deru T^r(u_0) \right \|_{\infty} \, \der r.
\end{align*}
Gronwall's lemma then yields
\begin{align*}
\left \| \Deru T^s (u_0) \right \|_{\infty} \leq 
\exp \left ( s\|f'\|_{[-\maxu,\maxu]} \right ).
\end{align*}
We thus have
\begin{align}\label{res_gen_order_lie1}
 \left \| T^tu_0 - X^t Y^t  u_0  \right \|_{\infty} \leq &
\int_0^t \int_0^s
\exp \left ( (t-s) \|f'\|_{[-\maxu,\maxu]}  \right )
\|f''\|_{[-\maxu,\maxu]} 
\left \| \partial_x  X^r Y^s u_0  \right \|^2_{\infty} 
\, \der r\, \der s.
\end{align}
Taking into account that
$
 \left \| \partial_x   X^r  Y^s u_0   \right \|^2_{\infty} =
\left \| X^r \partial_x  Y^s u_0   \right \|^2_{\infty} 
\leq
\left \| \partial_x  Y^s u_0   \right \|^2_{\infty} 
$
and with \eqref{derx_ivdure}, we finally obtain
\begin{align*}
 \left \| T^tu_0 - X^tY^t  u_0  \right \|_{\infty}   
& \leq 
\int_0^t \int_0^s
 \exp  \left ( (t+s) \|f'\|_{[-\maxu,\maxu]} \right ) \|f''\|_{[-\maxu,\maxu]}
\left \|   \partial_x  u_0   \right \|^2_{\infty}
\, \der r\, \der s \nonumber \\[0.5ex]
& \leq 
\frac{t^2 \exp \left ( 2t {\|f'\|_{[-\maxu,\maxu]}} \right )
\|f''\|_{[-\maxu,\maxu]} }{2}
\left \| 
\partial_x u_0 
\right \|^2_{\infty},
\end{align*}
which proves \eqref{lie1_order_2}.
The proof for \eqref{lie2_order_2} follows the same procedure
which yields
\begin{align}
 \left \| T^tu_0 - Y^t X^t u_0  \right \|_{\infty}  \leq &
\int_0^t \int_0^s
\exp \left ( (t+r) \|f'\|_{[-\maxu,\maxu]}  \right )
\|f''\|_{[-\maxu,\maxu]} 
\left \| \partial_x  X^s u_0  \right \|^2_{\infty} 
\, \der r\, \der s \nonumber \\[0.5ex]
\leq & 
\int_0^t 
\exp \left ( 2t \|f'\|_{[-\maxu,\maxu]}  \right )
\|f''\|_{[-\maxu,\maxu]}
\left \| \partial_x  X^s u_0  \right \|^2_{\infty} s 
\, \der s, \label{res_gen_order_lie2}
\end{align}
and proves \eqref{lie2_order_2} by considering that
$\left \| \partial_x  X^s u_0   \right \|^2_{\infty} =
\left \| X^s \partial_x   u_0   \right \|^2_{\infty}
\leq
\left \| \partial_x  u_0   \right \|^2_{\infty}$.
\end{proof}

Notice that both Lie schemes are bounded by the same
expression in Theorem \ref{theo_order_class}, and
for sufficiently small $t$
these estimates
involve the classical second order local error for
Lie splitting.
Considering now the Gauss-Weierstrass formula for the 
heat semi-group associated with \eqref{diff}, and the
Young's inequality for convolutions,
we have for all $u_0$ in $\BOinf$ and
$t>0$, the following regularizing effect of the Laplace
operator:
\begin{equation}\label{reg_lap}
\left \| \partial_x  X^t u_0   \right \|_{\infty}
\leq
\frac{1}{\sqrt{\pi t}} \|u_0\|_{\infty}.
\end{equation}
The following bounds can be then derived.
\begin{theorem}\label{theo_order_red}
For $t \in [0,{\rm T})$  and $u_0$ in $\BOinf$,
with
$
\maxu = \max ( \| u_0  \|_{\infty},1),
$
we have
\begin{equation}\label{lie1_order_1.5}
\left \| 
T^tu_0 -  X^t Y^t u_0
\right \|_{\infty} 
\leq
\frac{4\maxu\, t\sqrt{t} \exp \left ( t {\|f'\|_{[-\maxu,\maxu]}} \right )
\|f''\|_{[-\maxu,\maxu]} }{3\sqrt{\pi}}
\left \| 
\partial_x u_0 \right \|_{\infty} 
\end{equation}
\begin{equation}\label{lie2_order_1.5}
\left \| 
T^tu_0 -  Y^t X^t u_0
\right \|_{\infty} 
\leq
\frac{2\, t\sqrt{t} \exp \left ( 2t {\|f'\|_{[-\maxu,\maxu]}} \right )
\|f''\|_{[-\maxu,\maxu]} }{3\sqrt{\pi}} 
\left \| u_0 \right \|_{\infty}
\left \| \partial_x u_0 \right \|_{\infty},
\end{equation}
and
\begin{equation}\label{lie2_order_1}
\left \| 
T^tu_0 -  Y^t X^t u_0
\right \|_{\infty} 
\leq
\frac{t \exp \left ( 2t {\|f'\|_{[-\maxu,\maxu]}} \right )
\|f''\|_{[-\maxu,\maxu]}}{\pi} 
\left \| u_0 \right \|_{\infty}^2.
\end{equation}
\end{theorem}
\begin{proof}
Proof of \eqref{lie1_order_1.5} comes from considering
the regularizing effect of the Laplacian \eqref{reg_lap}
and that
\begin{align*}
 \left \| \partial_x   X^r  Y^s u_0   \right \|^2_{\infty}  & \leq
\left \| \partial_x   X^r  Y^s u_0   \right \|_{\infty}
\left \| X^r \partial_x  Y^s u_0   \right \|_{\infty} \nonumber \\[0.5ex]
& \leq
\frac{
\max (\left \| u_0 \right \|_{\infty},1) \times
\exp \left ( s {\|f'\|_{[-\maxu,\maxu]}} \right )
\left \| \partial_x u_0 \right \|_{\infty}}
{\sqrt{\pi r}},
\end{align*}
into \eqref{res_gen_order_lie1},
where assumption \eqref{assumptiononf} has been considered
and hence \eqref{majo}
for $ \| Y^s u_0  \|_{\infty}$.
Similarly,
$$
 \left \| \partial_x   X^s u_0   \right \|^2_{\infty} \leq
\left \| \partial_x   X^s u_0   \right \|_{\infty}
\left \| X^s \partial_x  u_0   \right \|_{\infty} 
\leq
\frac{
\left \| u_0 \right \|_{\infty}
\left \| \partial_x u_0 \right \|_{\infty}}
{\sqrt{\pi s}},
$$
and
$$
 \left \| \partial_x   X^s u_0   \right \|^2_{\infty} \leq
\frac{
\left \| u_0 \right \|_{\infty}^2
}
{\pi s},
$$
yield, respectively,
\eqref{lie2_order_1.5} and \eqref{lie2_order_1}
into \eqref{res_gen_order_lie2}.
\end{proof}

Theorem \ref{theo_order_red} provides then with alternative
estimates for both Lie methods.
To summarize, and using the notation in \eqref{lie_localerr}
for a time $t$ set by a given splitting time step $\split$,
we have the following results for $u_0$ in $\BOinf$
and $\split>0$:
\begin{equation*}\label{errloc_min_lie1}
\left\|
\err^{\split}_{\lie_1} u_0 
\right\|_{\infty}
\propto \min \left(
\| \partial_x u_0 \|^2_{\infty} \split^2,
\max ( \| u_0 \|_{\infty},1) \times
\| \partial_x u_0 \|_{\infty} \split^{1.5}
\right),
\end{equation*}
and
\begin{equation*}\label{errloc_min_lie2}
\left\|
\err^{\split}_{\lie_2} u_0 
\right\|_{\infty}
\propto \min \left(
\| \partial_x u_0 \|^2_{\infty} \split^2,
\| u_0 \|_{\infty}
\| \partial_x u_0 \|_{\infty} \split^{1.5},
\| u_0 \|_{\infty}^2
\split
\right).
\end{equation*}
For sufficiently small time steps, we thus recover the classical
second order $\Or(\split^2)$ for local errors in accordance
to the asymptotic behavior of both Lie splitting schemes.
For larger time steps, however,
the alternative estimates that behave like
$\Or(\split^{1.5})$ or $\Or(\split)$ might become more relevant,
which 
apparently entail a loss of accuracy of the splitting
approximations.
It is nevertheless important to highlight the impact of
the multiplying constants in the different estimates,
and in particular 
the nature of
the initial condition and its derivatives,
especially for non-asymptotic regimes defined by
relatively large time step.
Considering, for instance, solutions with 
high spatial gradients, 
the multiplying factor for
the classical estimates in $\split^2$ 
is of order
$\Or \left( \| \partial_x u_0 \|^2_{\infty} \right)$,
of smaller order 
$\Or \left(
\max ( \| u_0  \|_{\infty},1) \times
\| \partial_x u_0 \|_{\infty} \right)$
for $\split^{1.5}$, 
and of potentially much smaller 
$\Or \left(\| u_0 \|_{\infty}^2 \right)$
for $\split$.
Therefore, 
the alternative bounds
given in Theorem \ref{theo_order_red}
should describe much better the numerical behavior of the approximations,
which in this case yield
smaller local errors 
than those predicted by the classical estimates in $\split^2$,
initially derived in Theorem \ref{theo_order_class}.
The same discussion is valid for the Strang local error estimates
detailed in Appendix \ref{Strang_local_error}.

\section{Application to traveling waves}\label{sec:KPP}
In this part
we evaluate the previous theoretical study in the context
of reaction-diffusion problems that admit self-similar traveling wave solutions.
The main 
interest
of considering this kind of configuration
is 
that the featured stiffness can be tuned
using a space-time scaling.
Therefore, it provides the right framework to perform
a complete numerical validation
of the theoretical local error estimates.
Moreover,
a detailed study can be conducted on the impact of the stiffness featured  
by propagating fronts with steep spatial gradients as performed, for instance,
in \cite{Duarte11_parareal}.
In what follows, we first recast previous estimates in the context
of reaction traveling waves, and
then we will illustrate them 
with
the numerical 
solution of a KPP model.

\subsection{Theoretical estimates}
We consider the propagation of self-similar
waves modeled by parabolic PDEs of type: 
\begin{equation}\label{eq5:sys_rea_dif_wave}
\left.
\begin{array}{ll}
\partial_t u-D\, \partial^2_{x} u =kf(u),& \quad x\in \R,\ t>0,\\[0.5ex]
u(0,x)=u_0(x),&\quad x\in \R,
\end{array}\right\}
\end{equation}
with solution $u(x,t)=u_0(x - ct)$,
where $c$ is the steady speed of the wavefront,
and $D$ and $k$ stand, respectively, for
diffusion and reaction coefficients.
Introducing the Lie splitting solutions 
\eqref{Lie1}
for equation \eqref{eq5:sys_rea_dif_wave}
and taking into account that the corresponding 
Lie bracket is now defined as 
$\lbrace D \Delta ,k f \rbrace =
kD \lbrace \Delta , f \rbrace$,
we obtain the following estimates.
\begin{corollary}\label{theo_order_class_kD}
For $t \in [0,{\rm T})$  and $u_0$ in $\BOinf$,
with
$
\maxu = \max ( \| u_0  \|_{\infty},1),
$
we have
\begin{equation*}\label{lie1_order_2_kD}
\left \| 
T^tu_0 -  X^t Y^t u_0
\right \|_{\infty} 
\leq
\frac{ k D \, t^2\exp \left ( 2kt {\|f'\|_{[-\maxu,\maxu]}} \right )
\|f''\|_{[-\maxu,\maxu]} }{2 }
\left \| 
\partial_x u_0 
\right \|^2_{\infty},
\end{equation*}
and
\begin{equation*}\label{lie2_order_2_kD}
\left \| 
T^tu_0 - Y^t X^t  u_0
\right \|_{\infty} \leq
\frac{ k D \, t^2 \exp \left ( 2kt {\|f'\|_{[-\maxu,\maxu]}} \right )
\|f''\|_{[-\maxu,\maxu]} }{2}
\left \| 
\partial_x u_0   
\right \|^2_{\infty}.
\end{equation*}
\end{corollary}
Furthermore,
with the regularizing effect of the Laplacian:
$$\| \partial_x  X^t u_0   \|_{\infty}
\leq 
\frac{1}{\sqrt{\pi D t}}\|u_0\|_{\infty},$$
the following bounds can be derived.
\begin{corollary}\label{theo_order_red_kD}
 For $t \in [0,{\rm T})$  and $u_0$ in $\BOinf$,
with
$
\maxu = \max ( \| u_0  \|_{\infty},1),
$
we have
\begin{equation*}\label{lie1_order_1.5_kD}
\left \| 
T^tu_0 -  X^t Y^t u_0
\right \|_{\infty} 
\leq
\frac{4 \maxu \, k D \, t\sqrt{t} \exp \left ( kt {\|f'\|_{[-\maxu,\maxu]}} \right )
\|f''\|_{[-\maxu,\maxu]} }{3\sqrt{\pi D}}
\left \| 
\partial_x u_0 \right \|_{\infty} 
\end{equation*}
\begin{equation*}\label{lie2_order_1.5_kD}
\left \| 
T^tu_0 -  Y^t X^t u_0
\right \|_{\infty} 
\leq
\frac{2kD\, t\sqrt{t} \exp \left ( 2kt {\|f'\|_{[-\maxu,\maxu]}} \right )
\|f''\|_{[-\maxu,\maxu]} }{3\sqrt{\pi D}} 
\left \| u_0 \right \|_{\infty}
\left \| 
\partial_x u_0 \right \|_{\infty},
\end{equation*}
and
\begin{equation*}\label{lie2_order_1_kD}
\left \| 
T^tu_0 -  Y^t X^t u_0
\right \|_{\infty} 
\leq
\frac{k\, t \exp \left ( 2kt {\|f'\|_{[-\maxu,\maxu]}} \right )
\|f''\|_{[-\maxu,\maxu]} }{\pi }
\left \| u_0 \right \|_{\infty}^2.
\end{equation*}
\end{corollary}

In the context of traveling wave solutions,
the diffusion and reaction coefficients, $D$ and $k$, might be seen as 
scaling coefficients in time and space.
A dimensionless analysis of a traveling wave
can be then conducted,
as shown in \cite{Scott94},
by considering dimensionless time $\tau$ and space $r$: 
\begin{equation*}
 \tau = kt, \qquad r=(k/D)^{1/2}x.
\end{equation*}
As a consequence,
a steady velocity of the wavefront
can be derived
\begin{equation}\label{eq5:wave_speed}
\ds c=\der _t x \propto (Dk)^{1/2},
\end{equation}
whereas
the sharpness of the wave profile is measured
by 
\begin{equation}\label{eq5:wave_grad}
\ds \der _x u_0 \vert_{\max} = 
\left \| \partial_x u_0 \right \|_{\infty} \propto (k/D)^{1/2}.
\end{equation}
Condition $Dk=1$ then involves waves of constant velocity,
but greater $k$ (or 
smaller $D$) yields wavefronts with higher spatial gradients, and thus stiffer configurations.

By considering the Lie local errors
\eqref{lie_localerr}
for a time $t$ set by a given splitting time step $\split$,
the bounds from Corollary \ref{theo_order_class_kD} and \ref{theo_order_red_kD},
and the measure of the wave gradient \eqref{eq5:wave_grad}
together with condition  $Dk=1$,
we have that
for $u_0$ in $\BOinf$
and $\split>0$:
\begin{equation*}\label{errloc_min_lie1_kD}
\left \|
\err^{\split}_{\lie_1} u_0 
\right\|_{\infty}
\propto \min \left(
k^2 \split^2,
 \max ( \| u_0  \|_{\infty},1) \times
k^{1.5}  \split^{1.5}
\right),
\end{equation*}
and
\begin{equation*}\label{errloc_min_lie2_kD}
 \left\|
 \err^{\split}_{\lie_2} u_0 
 \right\|_{\infty}
 \propto \min \left(
k^2 \split^2,
\| u_0  \|_{\infty}
k^{1.5}  \split^{1.5},
\| u_0 \|_{\infty}^2
k \split
\right).
\end{equation*}
Even though these bounds are not sufficient to 
determine precisely the various 
intervals of numerical behavior depending on
$\split$, {\it i.e.}, the actual time steps for which
each bound becomes relevant,
for solutions with high spatial gradients
it is very likely to start having 
transitions from one behavioral regime to another
even for small splitting time steps
of about $\| \partial_x u_0 \|^{-1}_{\infty}$ (or $k^{-1}$
following \eqref{eq5:wave_grad} with  $Dk=1$),
based on a simple comparison of the multiplying coefficients in the estimates.
Similar conclusions can be drawn from Strang local error estimates.
 
\subsection{Numerical illustration: 1D KPP equation}
Recalling the classical Kol\-mo\-go\-rov-Petrovskii-Piskunov
model \cite{KPP38} 
with $f(u)=u(1-u)$, we consider in this study
a higher order KPP nonlinearity with 
$f(u)=u^2(1-u)$ 
(often referred to as Zeldovich nonlinearity).
The description of the dimensionless model and the structure of the 
analytical solution for this case 
can be found, for instance, in \cite{Scott94},
where a theoretical analysis shows that 
in the case with  $D=k=1$, 
the velocity of the self-similar traveling wave is
$c=1/\sqrt{2}$ in (\ref{eq5:wave_speed}), and the maximum gradient value reaches 
$1/\sqrt{32}$ in (\ref{eq5:wave_grad}). 
Notice that for this KPP nonlinearity
there is a single isolated value of the speed for which
the front exists, 
contrary to the monostable,
classical KPP equation.
In particular the case
$f(u)=u^2(1-u)$ verifies 
assumption \eqref{assumptiononf}
on $f(u)$, considered in \S\ref{sec:err_RD}.
The key point of this illustration is  
that the velocity of the traveling wave is proportional to 
$(k\,D)^{1/2}$, whereas the maximum gradient is proportional to 
$(k/D)^{1/2}$. 
Hence, we consider the case $kD=1$ for which one may obtain steeper
gradients with the same speed of propagation.
\begin{figure}[!htbb]
\begin{center}
 \includegraphics[width=0.49\textwidth]{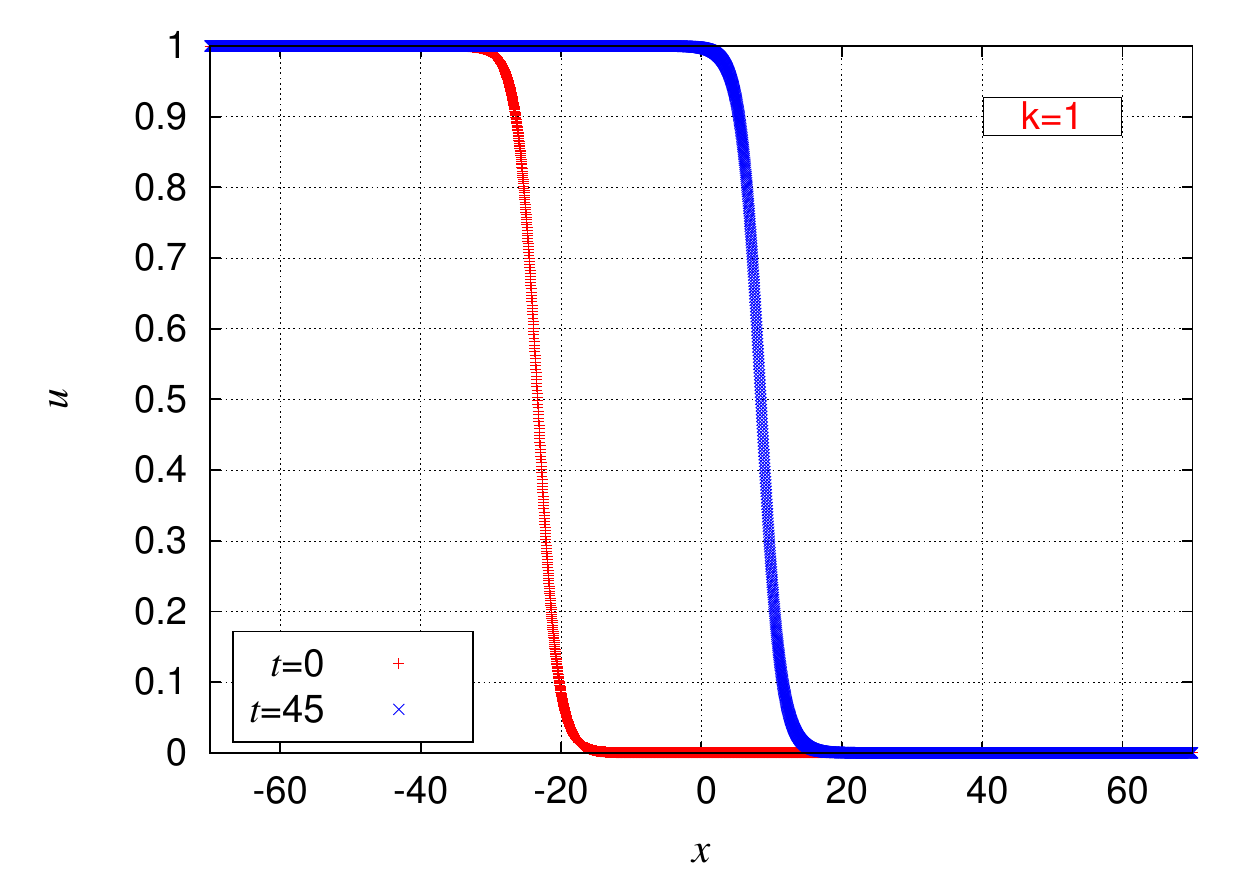}
 \includegraphics[width=0.49\textwidth]{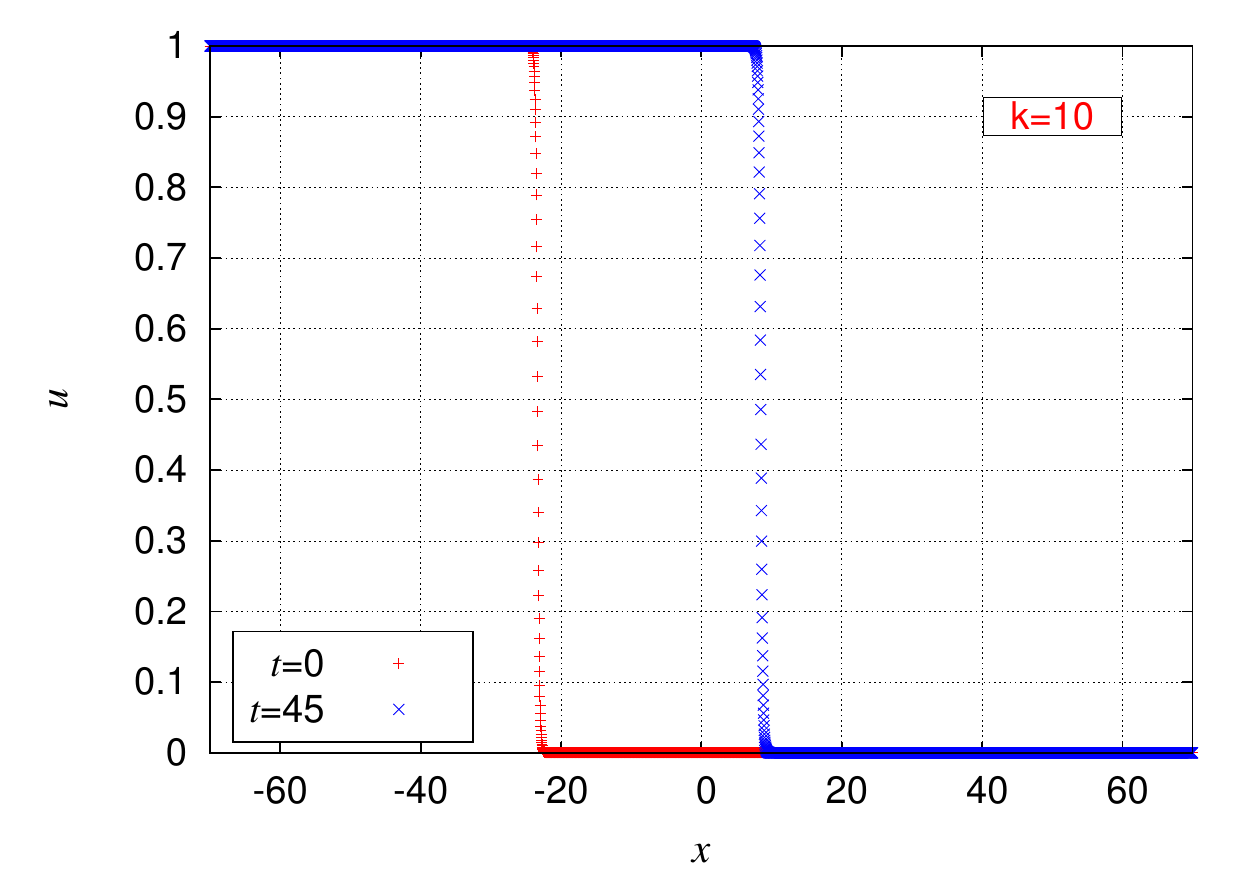}
 \includegraphics[width=0.49\textwidth]{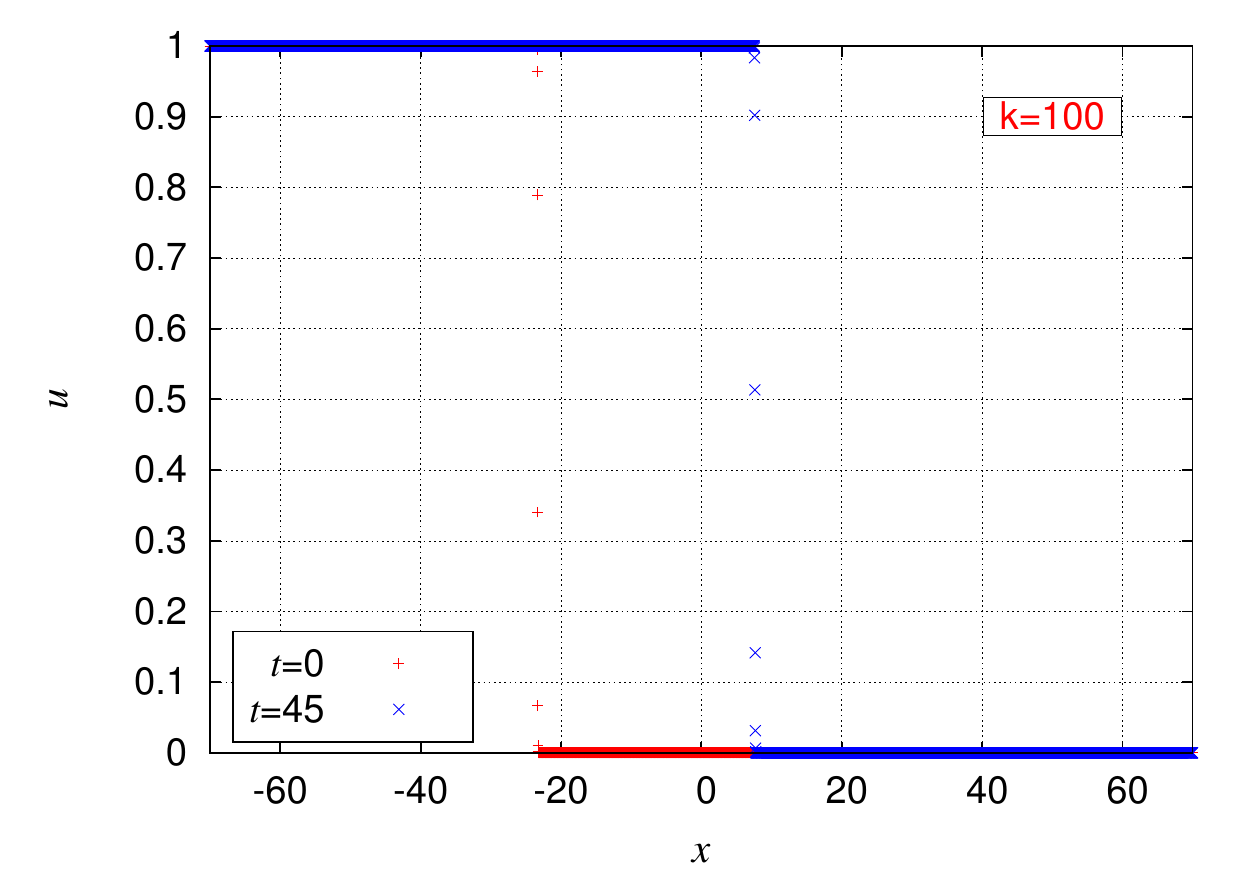}
 \includegraphics[width=0.49\textwidth]{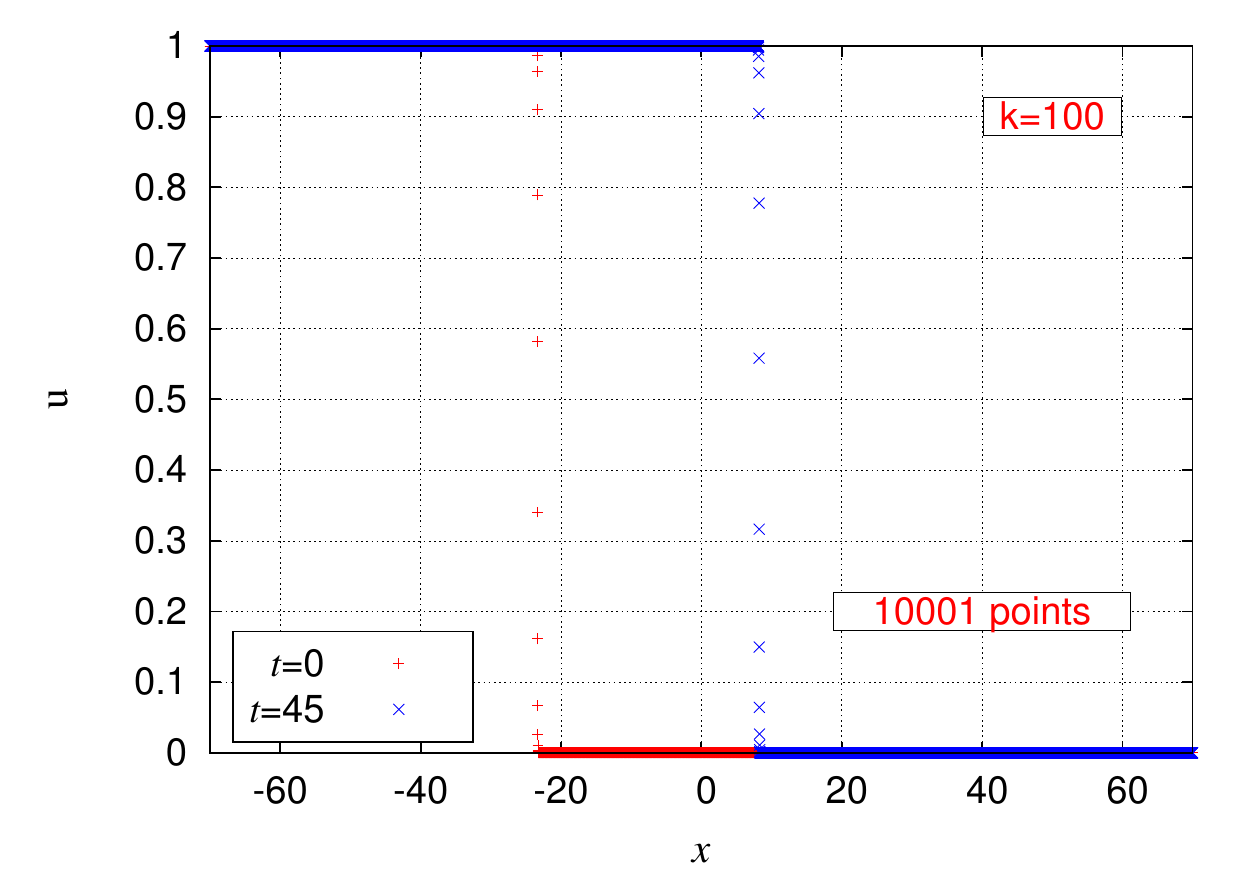}
\end{center}
\caption{1D KPP. Numerical {\it quasi-exact} solutions at $t=0$ and $t=45$
for $k=1$ (top left), $10$ (top right)
and $100$ (bottom left) with a discretization of $5001$ points.
Bottom right: case $k=100$ with $10001$ points
of discretization.
}
\label{fig6:sol_kpp}
\end{figure}
\begin{figure}[!htb]
\begin{center}
 \includegraphics[width=0.49\textwidth]{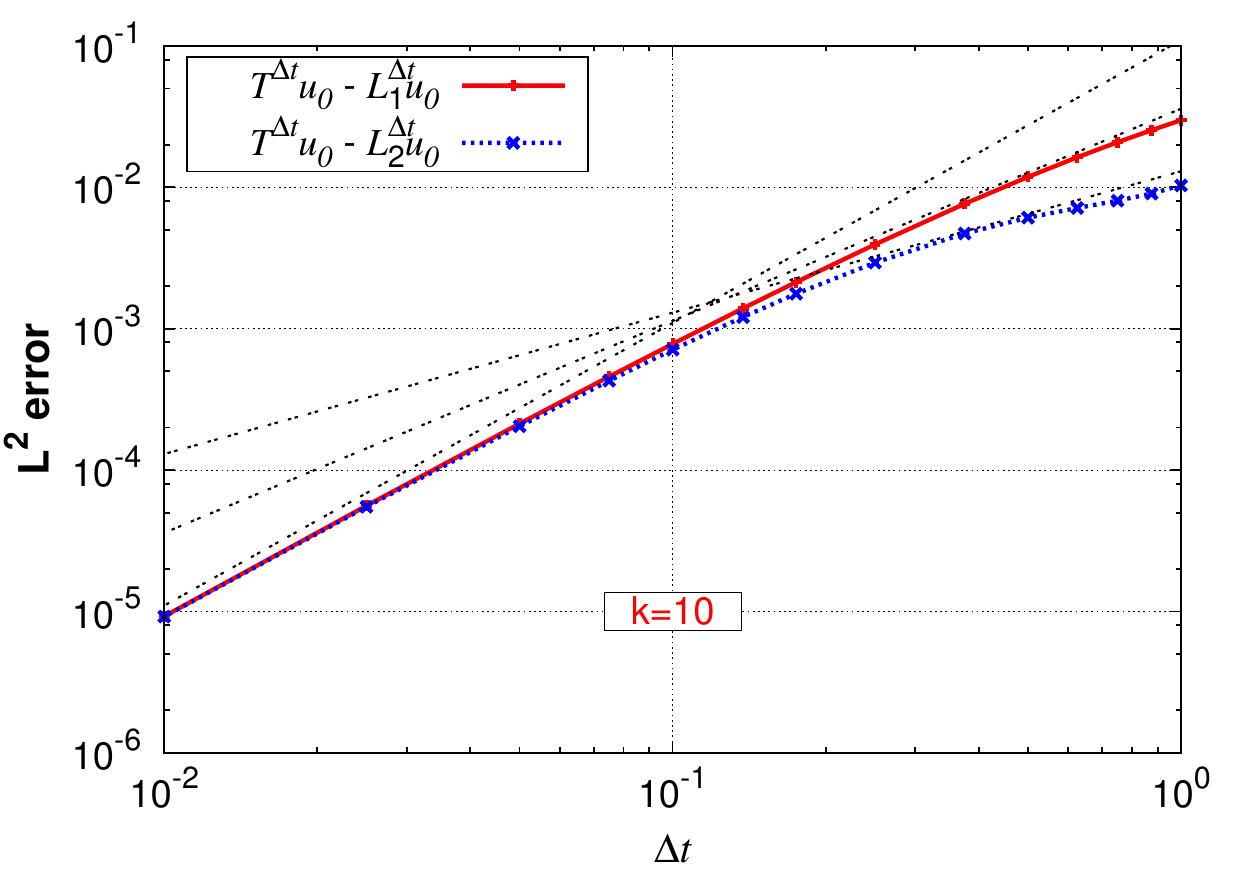}
 \includegraphics[width=0.49\textwidth]{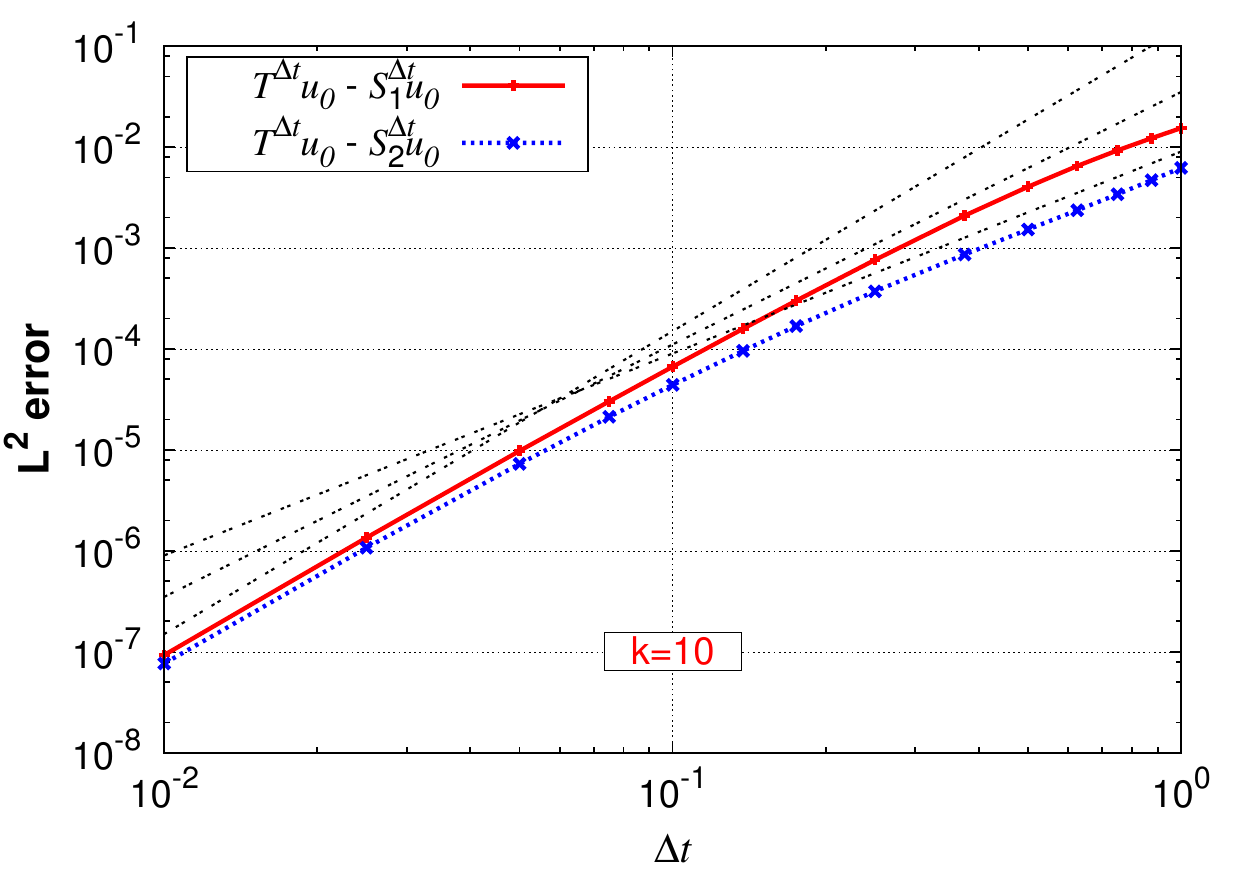}
 \end{center}
\caption{1D KPP with $k=10$. 
Local $L^2$ errors for several splitting time steps $\split$
for Lie (left) and Strang (right) splitting schemes.
Dashed lines with slopes $2$, $1.5$, and $1$ (left),
and with $3$, $2.5$, and $2$ (right)
are depicted.
}
\label{fig:kpp_k10}
\end{figure}
\begin{figure}[!htb]
\begin{center}
 \includegraphics[width=0.49\textwidth]{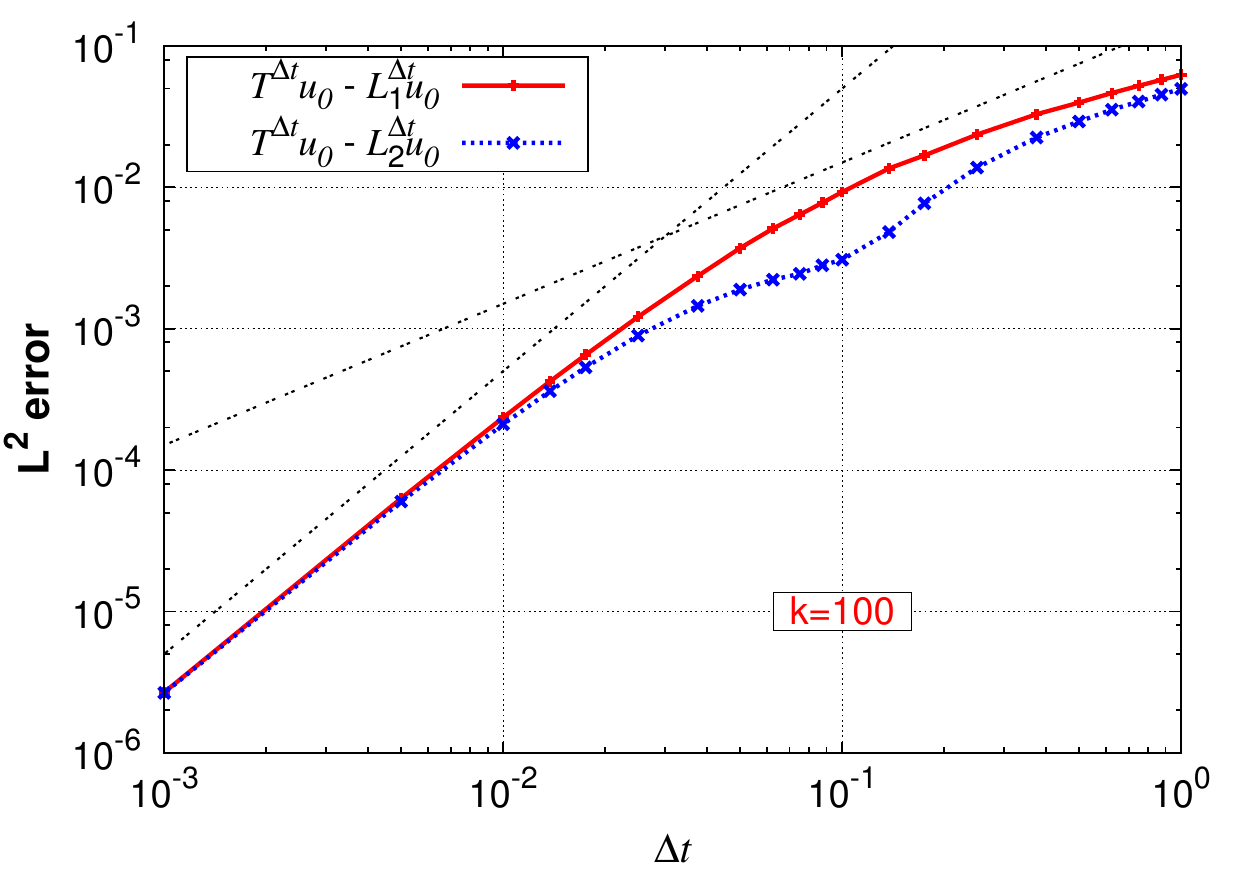}
 \includegraphics[width=0.49\textwidth]{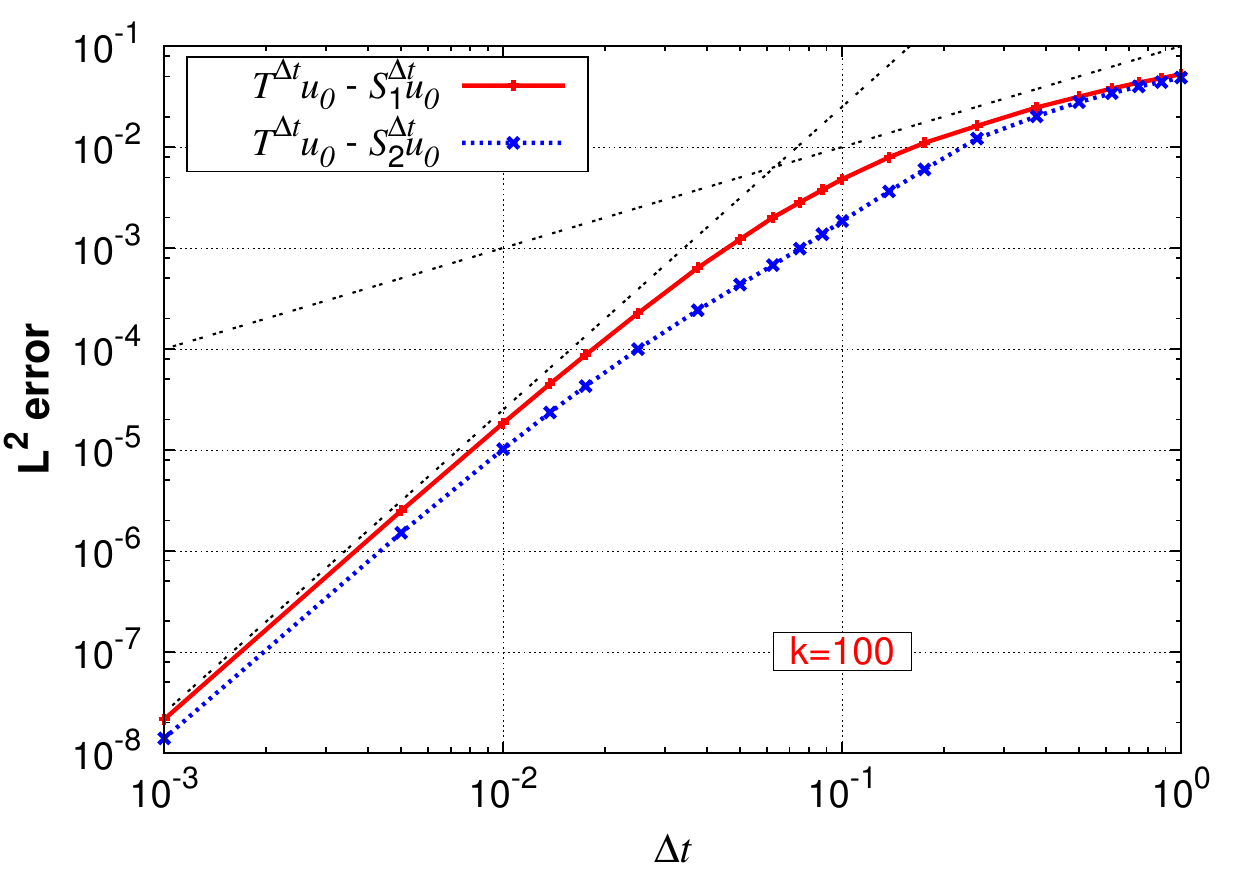}
 \end{center}
\caption{1D KPP with $k=100$. 
Local $L^2$ errors for several splitting time steps $\split$
for Lie (left) and Strang (right) splitting schemes.
Dashed lines with slopes $2$ and $1$ (left),
and with $3$ and $1$ (right)
are depicted.
}
\label{fig:kpp_k100}
\end{figure}
\begin{figure}[!htb]
\begin{center}
 \includegraphics[width=0.49\textwidth]{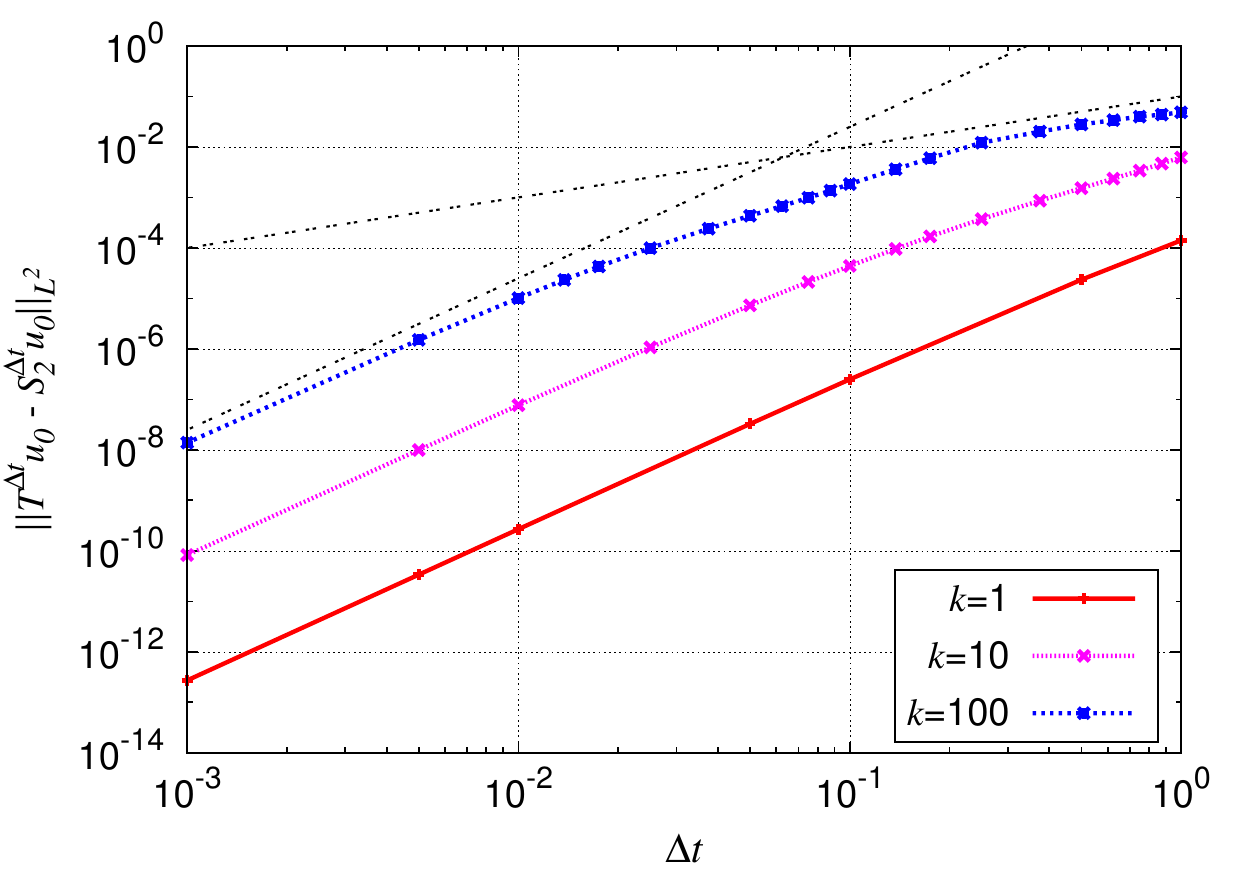}
 \includegraphics[width=0.49\textwidth]{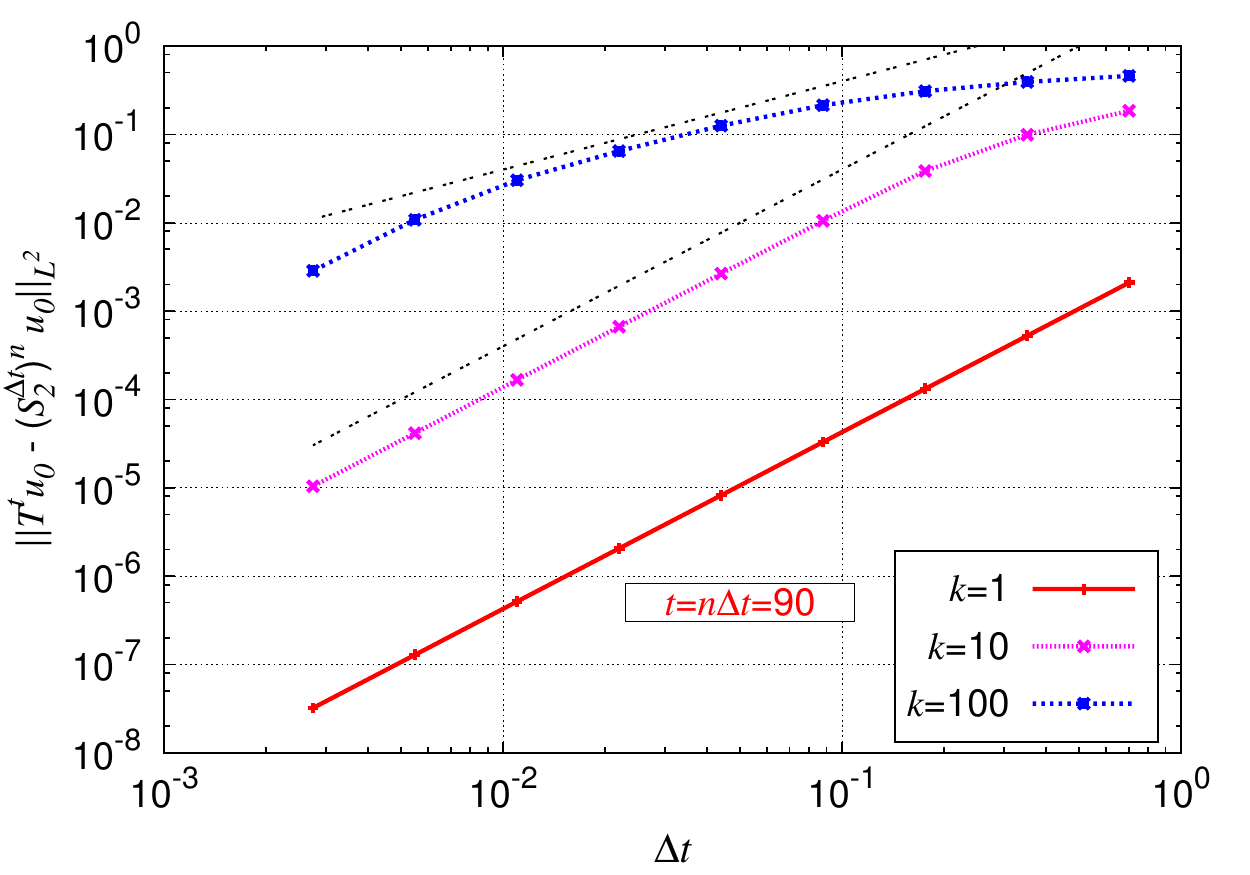}
 \end{center}
\caption{1D KPP. 
$\strang_2$-Strang local (left) and global (right) $L^2$ errors for several splitting time steps $\split$
and $k=1$, $10$, and $100$.
Dashed lines with slopes $3$ and $1$ (left),
and with $2$ and $1$ (right)
are depicted.
}
\label{fig:kpp_strang_loc_glob}
\end{figure}

For the numerical approximations, 
we consider a
1D discretization with $5001$ points 
over a region of $[-70,70]$ 
with homogeneous Neumann boundary conditions,
for which we have negligible spatial discretization errors with respect to
the ones coming from the numerical time integration
for the relatively large time steps that have been considered.
The Laplacian is discretized using a standard second order,
centered finite differences scheme.
The exact solution
$T^t u_0$ will be approximated by 
a reference or {\it quasi-exact} solution given by 
the numerical solution of the 
coupled reaction-diffusion equation
performed by the Radau5 solver \cite{Hairer96} with a fine tolerance,
$\eta_{\rm Radau5}=10^{-10}$.
Notice that even though an analytic solution exists, we consider
a reference solution corresponding
to the semi-discretized problem in order to avoid including
spatial discretization errors in the analysis,
that is, both the reference and splitting solutions
are computed on the same grid with the same spatial
discretization.
All splitting approximations
are computed with the splitting technique
introduced in \cite{Duarte11_SISC},
with Radau5 for the reactive term,
and the ROCK4 method \cite{MR1923724}
for the diffusion problem.
In order to properly discriminate splitting errors 
from those coming from the temporal integration of the subproblems, 
we consider the following fine tolerances, 
$\eta_{\rm Radau5}=\eta_{\rm ROCK4}=10^{-10}$.
For this particular problem another option for the splitting approximation
might have taken into account the ODE analytic solution for the
reaction steps, as well as the solution of the
discrete heat equation for the diffusion subproblems
by considering, for instance, Fast Fourier Transforms (FFT).
However, a fully numerical approach is adopted in this
study in accordance with more general and complex configurations
envisioned, as the ones presented in \cite{DuarteCFlame,Dumont2013}
and in the next section.
Figure \ref{fig6:sol_kpp} shows the numerical {\it quasi-exact} solutions
at times $t=0$ and $t=45$ for $k=1$, $10$, and $100$.
In what follows, $10001$  points of discretization are considered 
for $k=100$ instead of $5001$
in order to better represent the wavefront, as illustrated in Figure \ref{fig6:sol_kpp}.

We first compute the $L^2$ local errors 
for different splitting time steps $\split$
for all Lie
and Strang splitting schemes,
that is, $\lie^{\split}_1$,
$\lie^{\split}_2$,
$\strang^{\split}_1$,
and
$\strang^{\split}_2$ in
\eqref{Lie1} and \eqref{Strang}.
Starting from the same initial solution $u_0$,
the local error associated with $\lie^{\split}_1 u_0$
is measured by 
$\|T^{\split}u_0 - \lie^{\split}_1 u_0 \|_{L^2}$,
and similarly for the other schemes.
Figure \ref{fig:kpp_k10} illustrates these errors for $k=10$,
for relatively large splitting time steps.
A deviation from the asymptotic behavior
is exhibited for all splitting schemes for
time steps of order $k^{-1}$ or larger.
For smaller time steps we retrieve classical 
orders as established in Theorems \ref{theo_order_class}
and \ref{theo_order_class_strang}.
In this case,
$\lie _1$- and $\lie _2$-Lie schemes
are practically 
equivalent in terms of accuracy, as established in 
Theorem \ref{theo_order_class}.
On the other hand,
there is a slight difference for
$\strang _1$- and $\strang _2$-Strang schemes,
as seen in Theorem \ref{theo_order_class_strang}.
For the $\lie _1$-Lie scheme, 
the dependence on $\split$ varies
from $\split^2$ to $\split^{1.5}$,
whereas it attains $\split$
for the $\lie _2$-Lie scheme, as described in
Theorem \ref{theo_order_red}.
For the Strang schemes, 
the same phenomenon occurs 
from $\split^3$ to $\split^{2.5}$ and $\split^2$,
respectively, for the $\strang _1$-
and $\strang _2$-Strang schemes,
following the bounds established in Theorem \ref{theo_order_red_strang}.
Notice that in all cases a better accuracy is achieved
in the non-asymptotic regime
by splitting schemes ending with the reaction substep,
as previously proved in \cite{DesMas04}.
In particular the $\lie _2$-Lie scheme is even more 
accurate than a $\strang _1$-Strang one, for
sufficiently large splitting time steps.
Similar conclusions are drawn for a stiffer configuration
with $k=100$, illustrated in Figure \ref{fig:kpp_k100}.
In this case the splitting 
local errors eventually behave like $\split$.
In this way
the bounds introduced in Theorems 
\ref{theo_order_red} and \ref{theo_order_red_strang},
as well as 
the mathematical characterization of these errors
for non-asymptotic regimes,
consistently describe the 
numerical accuracy of operator splitting
for solutions disclosing high spatial gradients.
Considering the global error for the $\strang_2$-Strang scheme:
$\|T^{t}u_0 - (\strang^{\split}_2)^n u_0 \|_{L^2}$,
where 
$\strang^{\split}_2$ has been successively applied $n$ times to $u_0$,
such that $t = n\split$,
Figure \ref{fig:kpp_strang_loc_glob}
illustrates 
these errors
for the $\strang_2$-Strang scheme, which 
perfectly reproduces the
behavior of local errors.
The latter is not always the case since there might
be some error compensation and thus a global 
accuracy better
than the one theoretically expected.
This has been shown, for instance, in \cite{Ichinose01} 
for a linear configuration but
the proofs cannot be extended to a nonlinear framework.
In particular
the global error evaluation in
Figure \ref{fig:kpp_strang_loc_glob}
was made after a long integration time
in order to illustrate the worst possible configuration.
The influence of stiffness 
is highlighted for increasingly 
stiffer configurations corresponding to higher values of $k$.
Notice that for $k=1$, a non-stiff configuration,
asymptotic orders are preserved
even for relatively large splitting time steps.

\input Lowmach.tex

\section{Concluding remarks}\label{sec:conclusion}
We have introduced in this paper a rigorous
mathematical characterization of splitting errors
for nonlinear reaction-diffusion equations.
The corresponding error estimates are particularly relevant 
for relatively large splitting time steps,
and therefore for
many applications modeled by stiff PDEs
in which fast physical or numerical
scales usually impose prohibitively expensive
time steps.
In this context
splitting techniques can become a more efficient 
alternative to overcome
stability restrictions related to stiff source terms
or mesh size, as shown in \cite{Duarte11_SISC}. 
Additionally, a theoretical description of splitting errors
may also lead to further developments, as the adaptive
splitting scheme introduced in \cite{MR2847238}.
Understanding the numerical behavior of splitting
schemes, especially for relatively
large splitting time steps, is therefore shown to be
of the utmost importance
for both theoretical and practical reasons.
Besides,
we have illustrated the relevance of the present theoretical
study in the case of self-similar waves
with high spatial gradients.
This kind of problem mimics many other applications
characterized by the propagation of steep chemical fronts.
In particular we have considered a counterflow premixed
flame with 
complex chemistry and detailed transport,
for which we have also introduced a new way of implementing
operator splitting techniques.
In all cases
the key point of these numerical illustrations is that 
the present theoretical study consistently describes the behavior of the
numerical errors, especially for relatively large splitting time steps.
It can be thus seen how 
better accuracies 
are actually achieved
with respect to the asymptotic 
bounds
in the case of propagating fronts with steep spatial gradients.

\section*{Acknowledgements}
This research was 
supported by an ANR project grant (French National Research Agency - ANR Blancs):
{\it S\'echelles} (PI S. Descombes - 2009-2013),
by a DIGITEO RTRA project: \emph{MUSE} (PI M. Massot - 2010-2014),
and by a France-Stanford project (PIs P. Moin \& M. Massot - 2011-2012).
M.~Duarte was partially supported
by the Applied Mathematics Program of
the DOE Office of Advance Scientific Computing Research 
under contract No.~DE-AC02-05CH11231
at LBNL.

\appendix

\section{Local error estimates for Strang splitting}\label{Strang_local_error}
Based
on formula \eqref{eq1:strang_exact_error2} we can also obtain
an exact representation of Strang local errors,
considering the same type of computations
carried out for the proof of Theorem \ref{theo_exact_err}
and
taking into account
that
\begin{align}\label{comm_strang1}
  \left (
 \big[D_{f}, [D_{f},D_{\Delta}]\big]  \Id \right ) (u_0) =
&f^{(3)}(u_0)\left(\partial_x u_0,\partial_x u_0,f(u_0)\right)
+2 f''(u_0)\left(\partial_x u_0,f'(u_0)\partial_x u_0\right)
\nonumber \\
&-f'(u_0)f''(u_0)\left(\partial_x u_0,\partial_x u_0\right),
\end{align}
and
\begin{align}\label{comm_strang2}
 \left (
 \big[D_{\Delta}, [D_{\Delta},D_{f}]\big]  \Id \right ) (u_0)
=&
f^{(4)}(u_0) \left(\partial_x u_0,\partial_x u_0,\partial_x u_0,\partial_x u_0\right)
\nonumber \\
&+4f^{(3)}(u_0)\left(\partial_x u_0,\partial_x u_0,\partial^2_x u_0\right)
+2f''(u_0)\left(\partial^2_x u_0,\partial^2_x u_0\right).
\end{align}
\begin{theorem}\label{theo_exact_err_strang}
For $t \geq 0$ and $u_0$ in $\BOinf$, we have
\begin{align}
T^tu_0 - X^{t/2} Y^t X^{t/2} u_0  = &
- \frac{1}{4} \int_0^t \int_0^s (s-r)
\Deru T^{t-s} (X^{s/2} Y^s X^{s/2} u_0) X^{(s-r)/2} \, \times \nonumber \\[0.5ex]
& 
\left (
 \big[D_{\Delta}, [D_{\Delta},D_{f}]\big]  \Id \right ) 
 ( X^{r/2} Y^s X^{s/2} u_0 )
\, \der r\, \der s  \nonumber \\[0.5ex]
& + \frac{1}{2} \int_0^t \int_0^s (s-r)
\Deru T^{t-s} (X^{s/2} Y^s X^{s/2} u_0) X^{s/2} 
\, \times \nonumber \\[0.5ex]
&
\exp\left (\int_0^{r} f'(Y^{\sigma+s-r}  X^{s/2} u_0  )\, \der \sigma \right )
 \, \times \nonumber \\[0.5ex]
&  \left (
 \big[D_{f}, [D_{f},D_{\Delta}]\big]  \Id \right ) 
 (Y^{s-r}  X^{s/2} u_0)
\, \der r\, \der s
\label{strang1_exact_teo}
\end{align}
and
\begin{align}\label{strang2_exact_teo}
 T^tu_0 - Y^{t/2} X^t Y^{t/2} u_0  = &
- \frac{1}{4} \int_0^t \int_0^s (s-r)
\Deru T^{t-s} (Y^{s/2} X^s Y^{s/2} u_0) \, \times \nonumber \\[0.5ex]
&  
\exp\left (\int_0^{(s-r)/2} f'(Y^{\sigma+r/2}  X^s Y^{s/2} u_0  ) \der \sigma \right ) 
\, \times \nonumber \\[0.5ex]
&
\left (
 \big[D_{f}, [D_{f},D_{\Delta}]\big]  \Id \right )
 (Y^{r/2}  X^s Y^{s/2} u_0)
\, \der r\, \der s
\nonumber \\[0.5ex]
& + \frac{1}{2} \int_0^t \int_0^s (s-r)
\Deru T^{t-s} (Y^{s/2} X^s Y^{s/2} u_0) 
\, \times \nonumber \\[0.5ex]
&  
\exp\left (\int_0^{s/2} f'(Y^{\sigma}  X^s Y^{s/2} u_0  ) \der \sigma \right ) 
X^r 
\, \times \nonumber \\[0.5ex]
& 
\left (
 \big[D_{\Delta}, [D_{\Delta},D_{f}]\big]  \Id \right )
 (X^{s-r} Y^{s/2} u_0)
\, \der r\, \der s
\end{align}
\end{theorem}
The following bounds can be then obtained for 
the local errors corresponding to 
both Strang approximations \eqref{Strang},
following the procedure considered for the proof of
Theorem \ref{theo_order_class} for
\eqref{strang1_exact_teo}--\eqref{strang2_exact_teo}
together with 
\eqref{comm_strang1}--\eqref{comm_strang2}.
\begin{theorem}\label{theo_order_class_strang}
For $t \in [0,{\rm T})$  and $u_0$ in $\BOinf$,
with $\maxu = \max ( \| u_0 \|_{\infty},1)$,
we have
\begin{eqnarray*}\label{strang1_order_3}
&&\left \| 
T^tu_0 - X^{t/2} Y^t X^{t/2}
\right \|_{\infty} 
\leq \nonumber \\[0.5ex]
&& \hphantom{hh}
\left[ 
\frac{t^3 \|f^{(4)}\|_{[-\maxu,\maxu]}}{24} +
\frac{t^4 \|f^{(3)}\|_{[-\maxu,\maxu]}\| f'' \|_{[-\maxu,\maxu]}}{8} +
\frac{t^5 \| f'' \|^3_{[-\maxu,\maxu]}}{20}
\right]
\exp \left ( 4t {\|f'\|_{[-\maxu,\maxu]}} \right )
\left \| 
\partial_x u_0 
\right \|^4_{\infty} \nonumber \\[0.5ex]
&& \hphantom{hh}
+ \left[ 
\frac{t^3 \|f^{(3)}\|_{[-\maxu,\maxu]}}{6} +
\frac{t^4 \| f'' \|^2_{[-\maxu,\maxu]}}{8}
\right]
\exp \left ( 3t {\|f'\|_{[-\maxu,\maxu]}} \right )
\left \| 
\partial_x u_0 
\right \|^2_{\infty}
\left \| 
\partial^2_x u_0 
\right \|_{\infty} \nonumber \\[0.5ex]
&& \hphantom{hh}
+ 
\frac{t^3 
\exp \left ( 2t {\|f'\|_{[-\maxu,\maxu]}} \right )
\| f'' \|_{[-\maxu,\maxu]}}{12}
\left \| 
\partial^2_x u_0 
\right \|^2_{\infty}\nonumber \\[0.5ex]
&& \hphantom{hh}
+ 
\frac{t^3 
\exp \left ( 2t {\|f'\|_{[-\maxu,\maxu]}} \right )
\left[
\| f' \|_{[-\maxu,\maxu]}\| f'' \|_{[-\maxu,\maxu]} + \| f \|_{[-\maxu,\maxu]}\|f^{(3)}\|_{[-\maxu,\maxu]}
\right]
}{12}
\left \| 
\partial_x u_0 
\right \|^2_{\infty},
\end{eqnarray*}
and
\begin{eqnarray*}\label{strang2_order_3}
&&\left \| 
T^tu_0 - Y^{t/2} X^t Y^{t/2}
\right \|_{\infty} 
\leq \nonumber \\[0.5ex]
&& \hphantom{hh}
\left[ 
\frac{t^3 \|f^{(4)}\|_{[-\maxu,\maxu]}}{12} +
\frac{t^4 \|f^{(3)}\|_{[-\maxu,\maxu]}\| f'' \|_{[-\maxu,\maxu]}}{8} +
\frac{t^5 \| f'' \|^3_{[-\maxu,\maxu]}}{40}
\right]
\exp \left ( 2.5t {\|f'\|_{[-\maxu,\maxu]}} \right )
\left \| 
\partial_x u_0 
\right \|^4_{\infty} \nonumber \\[0.5ex]
&& \hphantom{hh}
+ \left[ 
\frac{t^3 \|f^{(3)}\|_{[-\maxu,\maxu]}}{3} +
\frac{t^4 \| f'' \|^2_{[-\maxu,\maxu]}}{8}
\right]
\exp \left ( 2t {\|f'\|_{[-\maxu,\maxu]}} \right )
\left \| 
\partial_x u_0 
\right \|^2_{\infty}
\left \| 
\partial^2_x u_0 
\right \|_{\infty} \nonumber \\[0.5ex]
&& \hphantom{hh}
+ 
\frac{t^3 
\exp \left ( 1.5t {\|f'\|_{[-\maxu,\maxu]}} \right )
\| f'' \|_{[-\maxu,\maxu]}}{6}
\left \| 
\partial^2_x u_0 
\right \|^2_{\infty}\nonumber \\[0.5ex]
&& \hphantom{hh}
+ 
\frac{t^3 
\exp \left ( 2t {\|f'\|_{[-\maxu,\maxu]}} \right )
\left[
\| f' \|_{[-\maxu,\maxu]}\| f'' \|_{[-\maxu,\maxu]} + \| f \|_{[-\maxu,\maxu]}\|f^{(3)}\|_{[-\maxu,\maxu]}
\right]
}{24}
\left \| 
\partial_x u_0 
\right \|^2_{\infty}.
\end{eqnarray*}
\end{theorem}
Considering the regularizing effects of the Laplacian
\eqref{reg_lap} and
\begin{equation*}\label{reg_lap_2}
\left \| \partial^2_x  X^t u_0   \right \|_{\infty}
\leq
\frac{1}{ t} \|u_0\|_{\infty},
\end{equation*}
the next theorem yields alternative estimates
as in Theorem \ref{theo_order_red} for
the Lie case.
\begin{theorem}\label{theo_order_red_strang}
For $t \in (0,{\rm T})$  and $u_0$ in $\BOinf$,
with $\maxu = \max ( \| u_0 \|_{\infty},1)$,
we have
\begin{align}
&\left \| 
T^tu_0 - X^{t/2} Y^t X^{t/2}
\right \|_{\infty} 
\leq \nonumber \\[0.5ex]
& \hphantom{hh}
\frac{t
\exp \left ( 4t {\|f'\|_{[-\maxu,\maxu]}} \right )
\left[
\| f^{(4)} \|_{[-\maxu,\maxu]}  \| u_0 \|^4_{\infty}
+ 4\pi \| f^{(3)} \|_{[-\maxu,\maxu]}  \| u_0 \|^3_{\infty}
+ 2\pi^2 \| f'' \|_{[-\maxu,\maxu]} \| u_0 \|^2_{\infty}
\right]
}{2\pi^2} 
\nonumber \\[0.5ex]
& \hphantom{hh}
+
\frac{t^2
\exp \left ( 4t {\|f'\|_{[-\maxu,\maxu]}} \right )
\left[
\|f^{(3)}\|_{[-\maxu,\maxu]}\| f'' \|_{[-\maxu,\maxu]}\| u_0 \|^4_{\infty}
+ \pi \| f'' \|^2_{[-\maxu,\maxu]}\| u_0 \|^3_{\infty}
\right]
}{\pi^2}
\nonumber \\[0.5ex]
& \hphantom{hh}
+ 
\frac{t^2
\exp \left ( 2t {\|f'\|_{[-\maxu,\maxu]}} \right )
\left[
\| f' \|_{[-\maxu,\maxu]}\| f'' \|_{[-\maxu,\maxu]} + \| f \|_{[-\maxu,\maxu]}\|f^{(3)}\|_{[-\maxu,\maxu]}
\right]\| u_0 \|^2_{\infty}
}{4\pi}
\nonumber \\[0.5ex]
& \hphantom{hh}
+
\frac{t^3 \exp \left ( 4t {\|f'\|_{[-\maxu,\maxu]}} \right )
\| f'' \|^3_{[-\maxu,\maxu]}
\| u_0 \|^4_{\infty}
}{3\pi^2}. \nonumber
\end{align}
and
\begin{align}
&\left \| 
T^tu_0 - Y^{t/2} X^t Y^{t/2}
\right \|_{\infty} 
\leq \nonumber \\[0.5ex]
& \hphantom{hh}
\frac{
\maxu \,
t\sqrt{t} 
\exp \left ( t {\|f'\|_{[-\maxu,\maxu]}} \right )
\left[
2 \maxu^2 \| f^{(4)} \|_{[-\maxu,\maxu]} 
+ 8\pi \maxu \| f^{(3)} \|_{[-\maxu,\maxu]} 
+ 4\pi \| f'' \|_{[-\maxu,\maxu]}
\right]
}{3\pi\sqrt{\pi}}
\left \| 
\partial_x u_0 
\right \|_{\infty}\nonumber \\[0.5ex]
& \hphantom{hh}
+ 
\frac{
\maxu^2 \,
t^2
\exp \left ( 1.5t {\|f'\|_{[-\maxu,\maxu]}} \right )
\left[
\| f' \|_{[-\maxu,\maxu]}\| f'' \|_{[-\maxu,\maxu]} + \| f \|_{[-\maxu,\maxu]}\|f^{(3)}\|_{[-\maxu,\maxu]}
\right]
}{16\pi}.
\nonumber
\end{align}
\end{theorem}
Notice that the previous bounds are derived
by using the regularizing effects of the Laplacian
as much as possible.
Additional bounds could be nevertheless derived 
(similar to \eqref{lie2_order_1.5} in
Theorem \ref{theo_order_red})
of $\Or(t^{2.5})$,
$\Or(t^{2})$, and $\Or(t^{1.5})$
for the $\strang_1$-Strang scheme,
and of $\Or(t^{2.5})$ and 
$\Or(t^{2})$ for the $\strang_2$-Strang splitting.

\input Annexe_CF.tex

%
\bibliographystyle{plain}

\bibliography{biblio_LM}

\end{document}

%% file: Lowmach.tex
\section{Application to the dynamics of premixed flames}\label{sec:complex}
We now consider the simulation of a 
counterflow premixed methane flame with detailed transport and complex chemistry
in the low Mach number regime. 
These flames have received an extensive number of studies in both the steady 
and the pulsated case  for realistic engineering applications
(see {\it e.g.}, \cite{smooke90,gao96,schuller06}). 
Here we will consider the configuration where the flame is pulsated periodically 
in time through upstream modulations \cite{darabiha92}.
In this way we consider a time dependent system of PDEs for which we
introduce a new way to perform operator splitting,
compatible with the low Mach constraint.
Taking into account that there is already
a validated coupled, fully implicit code, based on \cite{darabiha92,massot98},
that can provide us with the reference dynamics of such flames,
we evaluate the splitting errors introduced by this new approach
and analyze the resulting behavior based on the theoretical study previously
conducted.

\subsection{Governing equations}
We consider 
two premixed flames stabilized in a symmetric framework where 
two injections of methane-air mixture occur in a counterflow way
(see Figure \ref{fig:setup} in Appendix \ref{Annexe_CF}). 
Isobaric flames equations are considered in the low Mach Number 
limit \cite{majda85}, so that
for ${\rm x} \in \R^d$ 
the pressure reads $p(t,{\rm x})=p_{atm}+\tilde p(t,{\rm x})$, 
where $\tilde p$ is a perturbation of the atmospheric pressure.
The counterflow configuration admits a symmetry of revolution
and thus the set of equations can be written as a 2D axisymmetric system.
In particular, we consider 1D similarity solutions of this 2D system of equations
for which the density of the gas $\rho$,
the temperature $T$, 
the axial velocity $u_z$, 
the reduced radial velocity  $u_r/r$, 
and the mass fractions $Y_k$ of the gas species have no radial dependence,
and all of them are functions of the axial coordinate $z$. 
Assuming that
the perturbation on the atmospheric pressure field 
is given by $\tilde p = -J r^2/2+ \hat p(z)$, where $r$ denotes the radial coordinate,
the governing equations read
\begin{align}
 \label{syst-qe1}
 \rho c_p \partial_t T
+ c_p V \partial_z T
- \partial_z\left(\lambda
\partial_z T \right)
&=-\sum_{k \in S}h_km_k\omega_k
-\sum_{k \in S}\rho Y_k c_{p,k} \mathcal V_{z,k} \partial_z T,
\\
 \label{syst-qe2}
 \rho \partial_t Y_k
+ V \partial_z Y_k 
+ \partial_z 
\left(\rho Y_k  \mathcal V_{z,k}\right) &= m_k\omega_k, \qquad k \in S,
\\[1.5ex]
\label{syst-qe5}
 \partial_z J & =0,
 \\[1.5ex]
 \label{syst-qe3}
 \rho \partial_t U
+ \rho U^2
+V \partial_z U &=
J+ \partial_z 
\left(\mu \partial_z U \right),
\\[1.5ex] 
 \label{syst-qe4}
 \partial_t \rho + 2\rho U+
 \partial_z V &= 0,
\end{align}
where $V= \rho u_z $ is the axial mass flux, $U$ the reduced radial velocity, 
$S$ the set of species indices,
$c_p$ the specific heat of the gas mixture, $c_p = \sum_{k\in S}Y_k c_{p,k}$,
$c_{p,k}$ the specific heat of the $k$-th species,
$h_k$ its enthalpy, $m_k$ its molar mass, 
$\lambda$ the heat conductivity, 
$\mu $ the shear viscosity, 
$J$ the reduced pressure gradient, $\omega_{k}$ 
the molar chemical production rate, 
and $\mathcal V _{k,z}$ 
the axial diffusion velocity of the $k$-th gas species. 
Density $\rho$ 
is a function of the local temperature and gas composition
through the ideal gas state equation.
Full details on this model can be found, for instance, in \cite{MR1713516}.

Given the symmetry of this configuration,
only half domain is considered, $z\ge 0$, 
with symmetry conditions at $z=0$.
The top boundary at $z=1.55\,$cm coincides with the fuel
injection 
and thus 
fixed values of the temperature, the axial and the reduced radial velocities, 
and the gas composition are imposed.
Its velocity is of $1.423\,$m/s, 
pulsated with a modulation of $10$\% at a frequency of $100\,$Hz.
The gas is composed of methane with a mass fraction equal to $3.88$\%,
mixed with air at $293\,$K and atmospheric pressure.
A detailed methane-air chemical kinetic mechanism with 29 species and 150 reactions
is considered, whereas transport parameters are computed based on \cite{ern94}.

\subsection{Introduction of operator splitting}
We aim at solving separately the chemical sources 
in system \eqref{syst-qe1}--\eqref{syst-qe4}:
\begin{align}
 \label{chimie_contribution1}
\rho c_p \partial_t T
& ={-\sum_{k\in S}h_km_k\omega_k},\\
 \label{chimie_contribution2}
\rho \partial_t Y_k
&= {m_k\omega_k}, \qquad k\in S,
\end{align}
and then consider
the following convection-diffusion problem:
\begin{align}
\label{Eq:systemNoChimie1}
 \rho c_p \partial_t T + c_p V \partial_z T
- \partial_z \left(\lambda \partial_z T \right)
&=
-\sum_{k\in S}\rho Y_k c_{p,k} \mathcal V_{z,k} \partial_z T,\\
\label{Eq:systemNoChimie2}
 \rho \partial_t Y_k
+ V \partial_z Y_k 
+ \partial_z \left(\rho Y_k  \mathcal V_{z,k}\right)
&= 0, \qquad k \in S,\\[1.5ex]
\label{Eq:systemNoChimie4}
 \partial_z J&=0,
\\[1.5ex]
\label{Eq:systemNoChimie3}
 \rho \partial_t U
+ \rho U^2
+V \partial_z U& =
J + \partial_z \left(\mu \partial_z U
\right),\\[1.5ex]
\label{Eq:splitsystem}
  \partial_t \rho
+ 2\rho U+ \partial_z V &= 0.
\end{align}
In this way, we obtain a decoupled system of ODEs 
\eqref{chimie_contribution1}--\eqref{chimie_contribution2}
on each grid point of the domain,
for which a dedicated stiff ODE solver can be implemented;
whereas the numerical effort required to solve
the coupled system \eqref{Eq:systemNoChimie1}--\eqref{Eq:splitsystem}
is also relieved.
However, since density $\rho$ depends on the local temperature
and gas composition, its time variation during the chemistry
step \eqref{chimie_contribution1}--\eqref{chimie_contribution2}
must be taken into account in equation
\eqref{Eq:splitsystem}.
Deriving in time the ideal gas state equation and
considering \eqref{chimie_contribution1}--\eqref{chimie_contribution2}, 
this variation, denoted as $(\partial_t \rho)_{\rm chem}$,
is given by
\begin{equation}
\label{Eq:drhodtchimie}
 (\partial_t \rho)_{\rm chem}
= \frac{1}{c_p T} \sum_{k\in S}h_km_k\omega_k
-m \sum_{k\in S}\omega_k.
\end{equation}
Hydrodynamics are therefore solved, coupled with the transport
equations without chemical source terms
for temperature and species, in
system \eqref{Eq:systemNoChimie1}--\eqref{Eq:systemNoChimie3}
together with
\begin{equation}\label{eq:corr_dthodt}
  \partial_t \rho
+ (\partial_t \rho)_{\rm chem}
+ 2\rho U+ \partial_z V = 0,
\end{equation}
instead of \eqref{Eq:splitsystem}.

In this implementation,
the corrective term $(\partial_t \rho)_{\rm chem}$
is updated at the beginning of each splitting time 
step, and kept constant throughout the time integration
of the current time step.
Considering the instantaneous nature of this correction
that affects especially the solution of the hydrodynamics,
both Lie and Strang schemes should finish with the numerical
solution of the convection-diffusion problem
\eqref{Eq:systemNoChimie1}--\eqref{Eq:systemNoChimie3} plus
\eqref{eq:corr_dthodt}.
This is also coherent with the idea of always ending the splitting
scheme with the fastest operator \cite{Sport2000,Kozlov2004,DesMas04}.
The convection-diffusion system is numerically solved with
the same code
considered for the original full problem \eqref{syst-qe1}--\eqref{syst-qe4}.
The method considers implicit time integration of the coupled
equations on a dynamically adapted grid (see details in \cite{smooke90,darabiha92}).
On the other hand,
the chemical source terms \eqref{chimie_contribution1}--\eqref{chimie_contribution2}
are integrated point-wise with the Radau5 solver \cite{Hairer96}.

\subsection{Numerical results}
To visualize the numerical performance of the splitting approximation, 
the point $z=0.25\,$cm in the high gradient zone is chosen
(see Figure \ref{fig:profT}).
The evolution of the temperature is shown in Figure \ref{fig:T_evolve}
for Lie and Strang approximations with different splitting time steps.
The reference solution corresponds to the solution of the 
full problem \eqref{syst-qe1}--\eqref{syst-qe4},
computed with fine tolerances (see \cite{smooke90}).
For the time steps considered 
the dynamics of the flame is properly reproduced 
with the new operator splitting introduced.
The same can be observed even for 
minor species, as illustrated, for instance, in Figure \ref{fig:OH_evolve}
for $Y_{\rm OH}$.
\begin{figure}[!htb]
\centering
\includegraphics[width=0.9\textwidth]{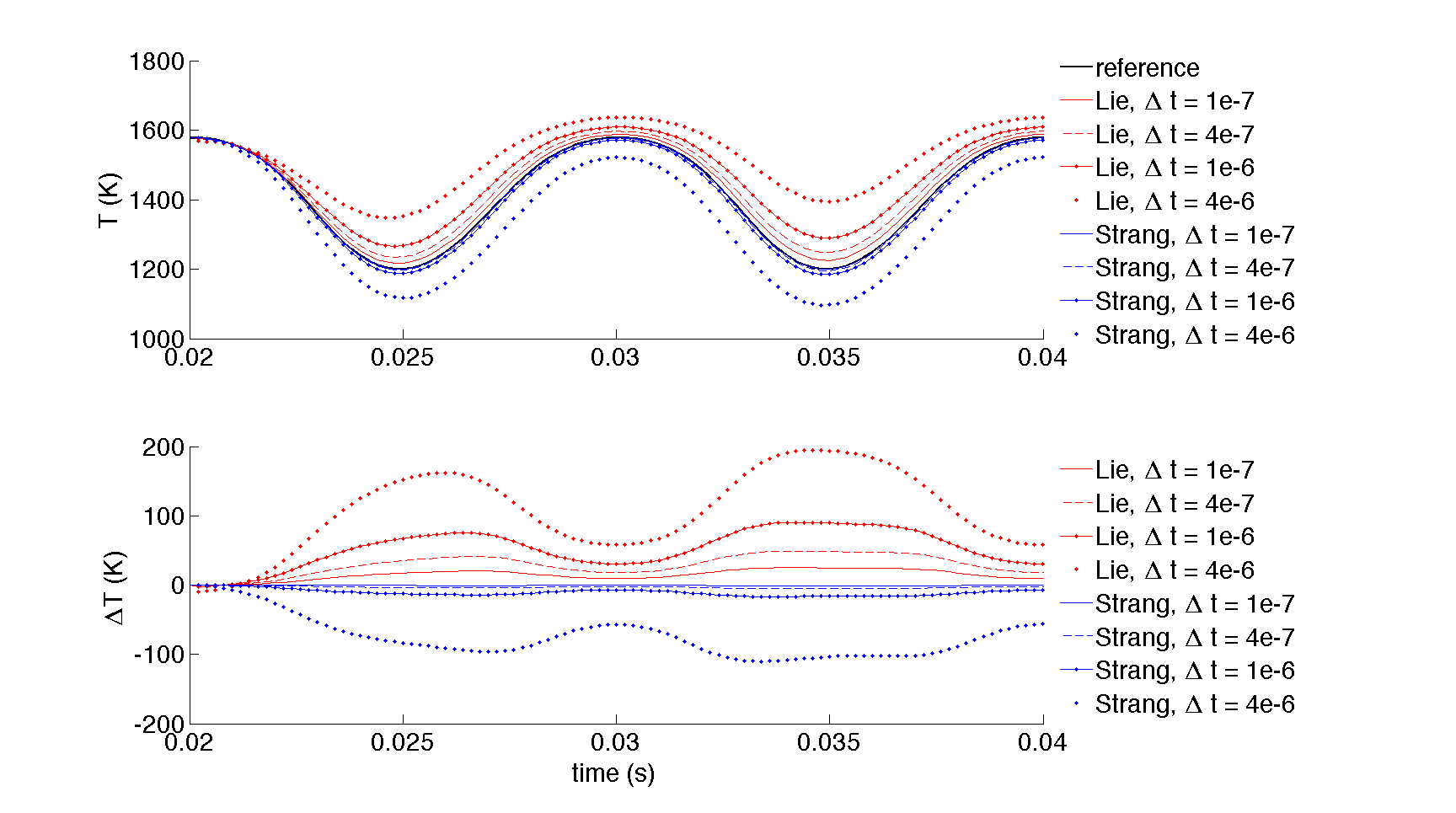}
\caption{Top: time evolution of temperature at point $z=0.25\,$cm,
with the reference solution (black line), and the Lie (blue lines) 
and Strang (red lines) splitting approximations. 
Bottom: difference with respect to the reference solution with $\Delta t =  10^{-7}$.} 
\label{fig:T_evolve}
\end{figure}
\begin{figure}[!htb]
\centering
\includegraphics[width=0.9\textwidth]{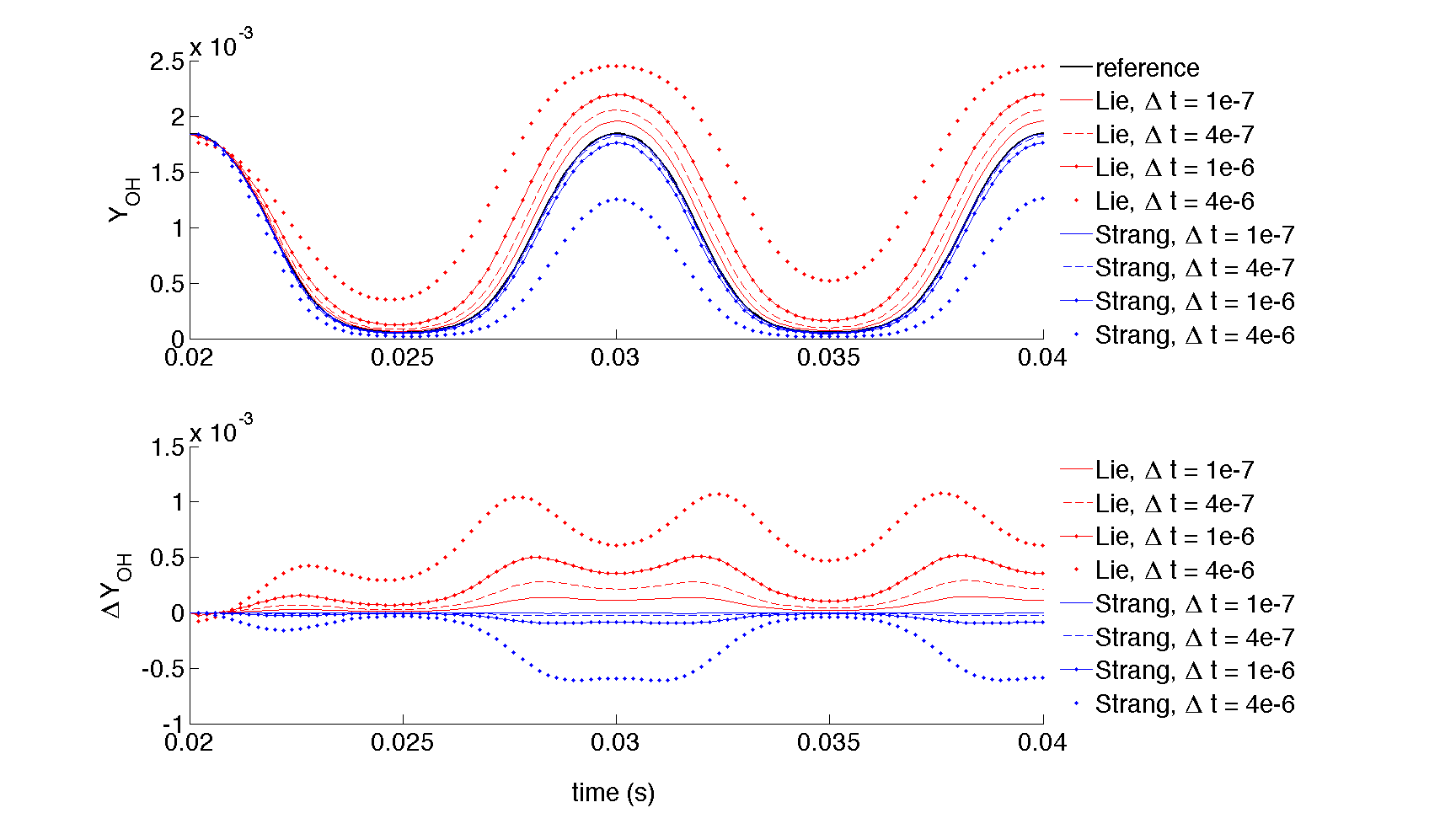}
\caption{Top: time evolution of the mass fraction of $OH$ at point $z=0.25\,$cm 
with the reference solution (black line), and the Lie (blue lines)
and Strang (red lines) splitting approximations. 
Bottom: difference with respect to the reference solution with $\Delta t =  10^{-7}$.} 
\label{fig:OH_evolve}
\end{figure}
\begin{figure}[!htb]
\centering
\includegraphics[width=0.49\textwidth]{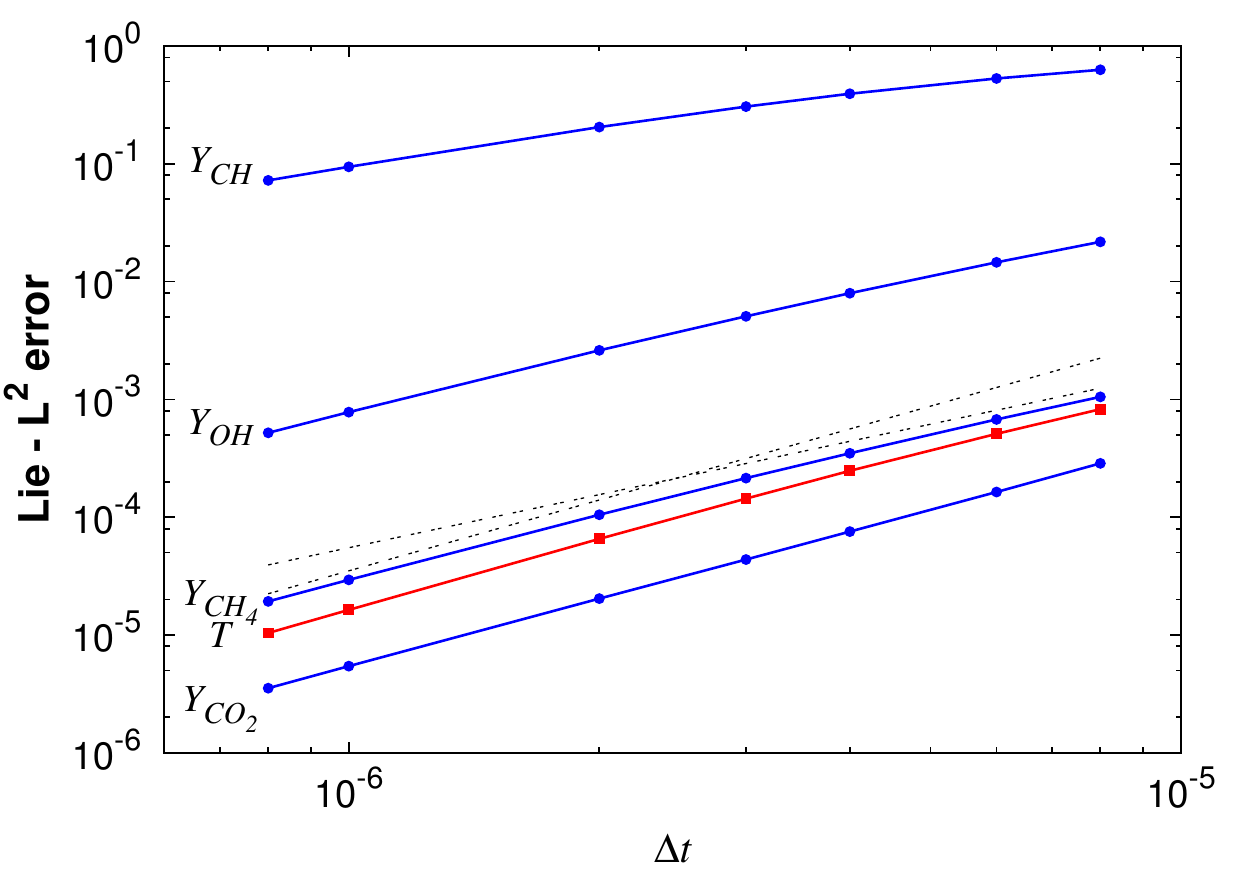} 
\includegraphics[width=0.49\textwidth]{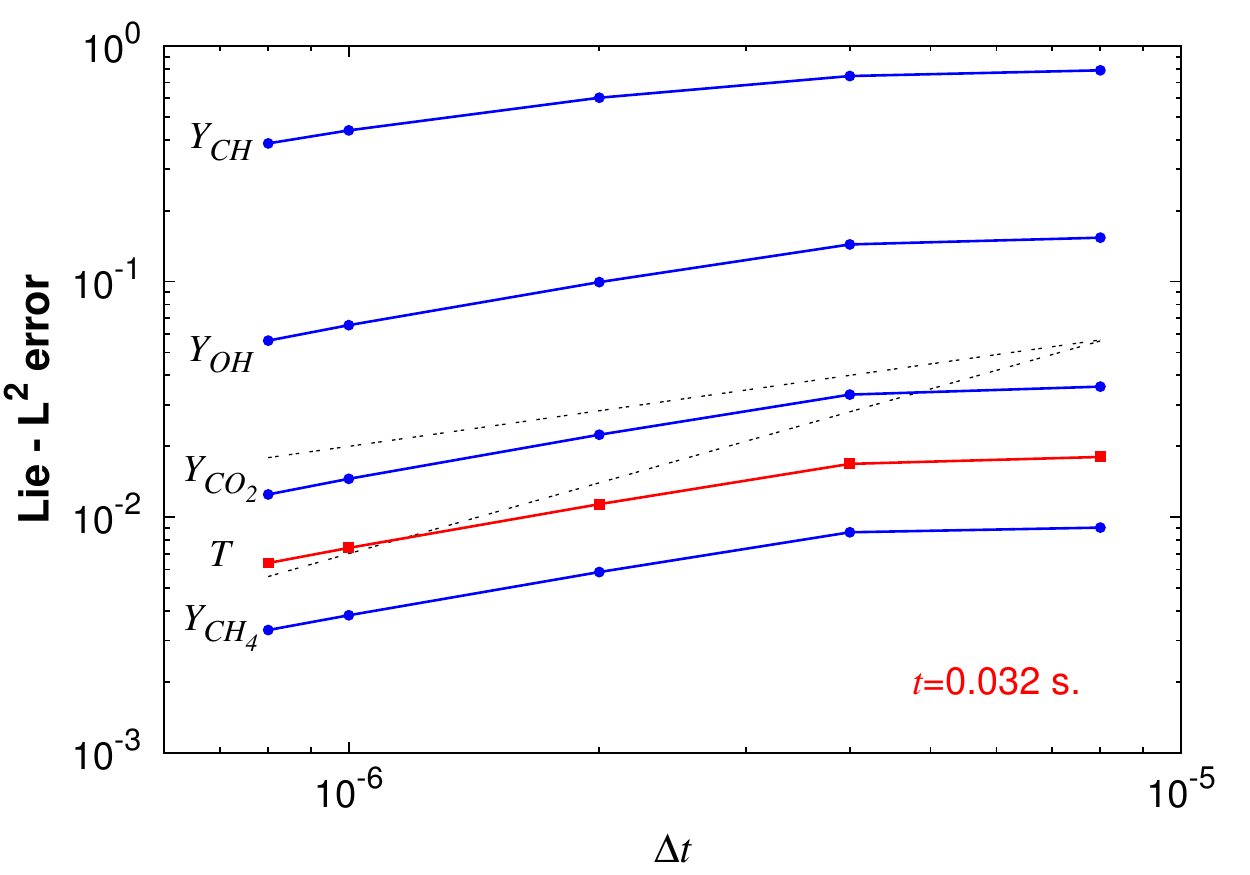} 
\caption{Local (left) and global (right) $L^2$ errors for the Lie
scheme for temperature $T$ and species $Y_{\rm CH_4}$, $Y_{\rm CO_2}$,
$Y_{\rm OH}$, and $Y_{\rm CH}$.
Lines with slope of 2 and 1.5 (left),
and of 1 and 0.5 (right) are depicted.}
\label{fig:lie_norms_CC}
\end{figure}
\begin{figure}[!htb]
\centering
\includegraphics[width=0.49\textwidth]{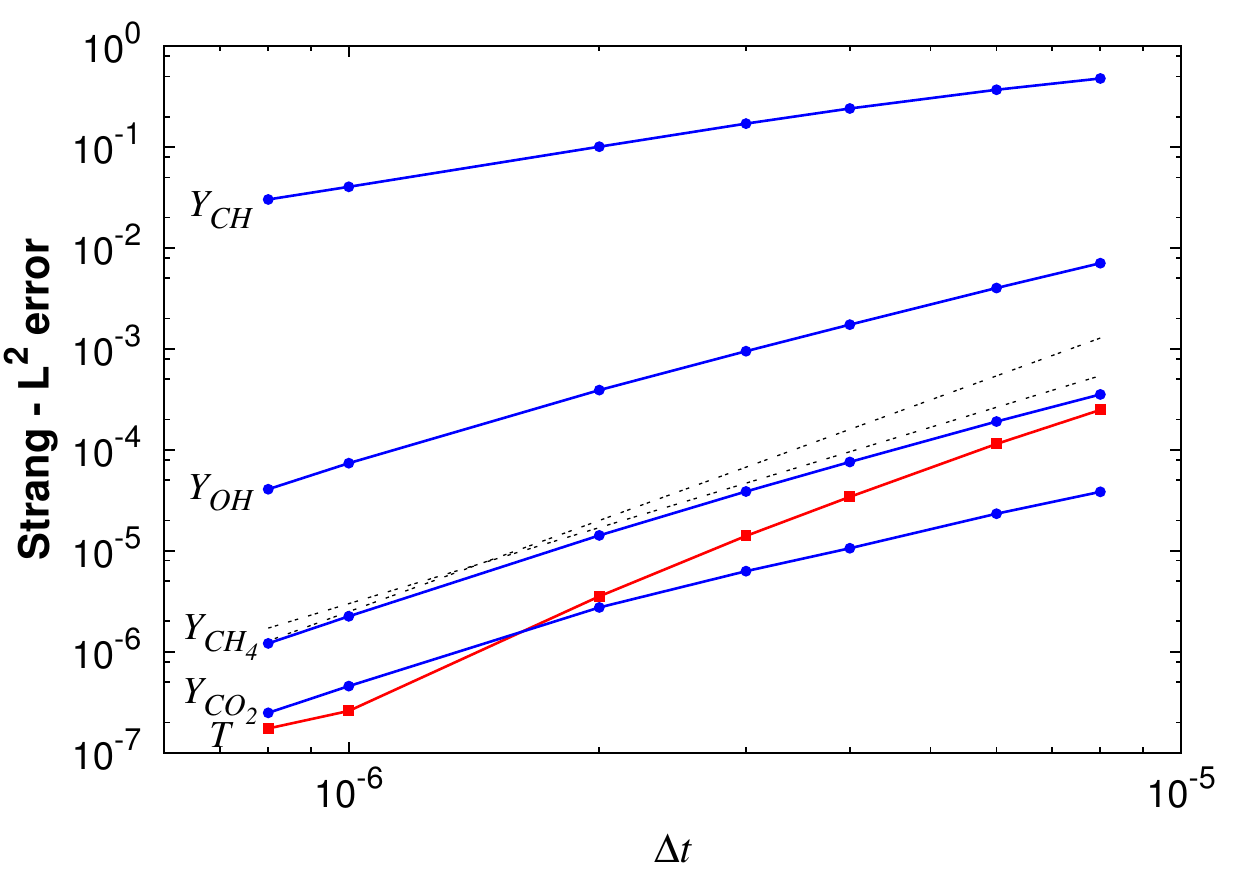} 
\includegraphics[width=0.49\textwidth]{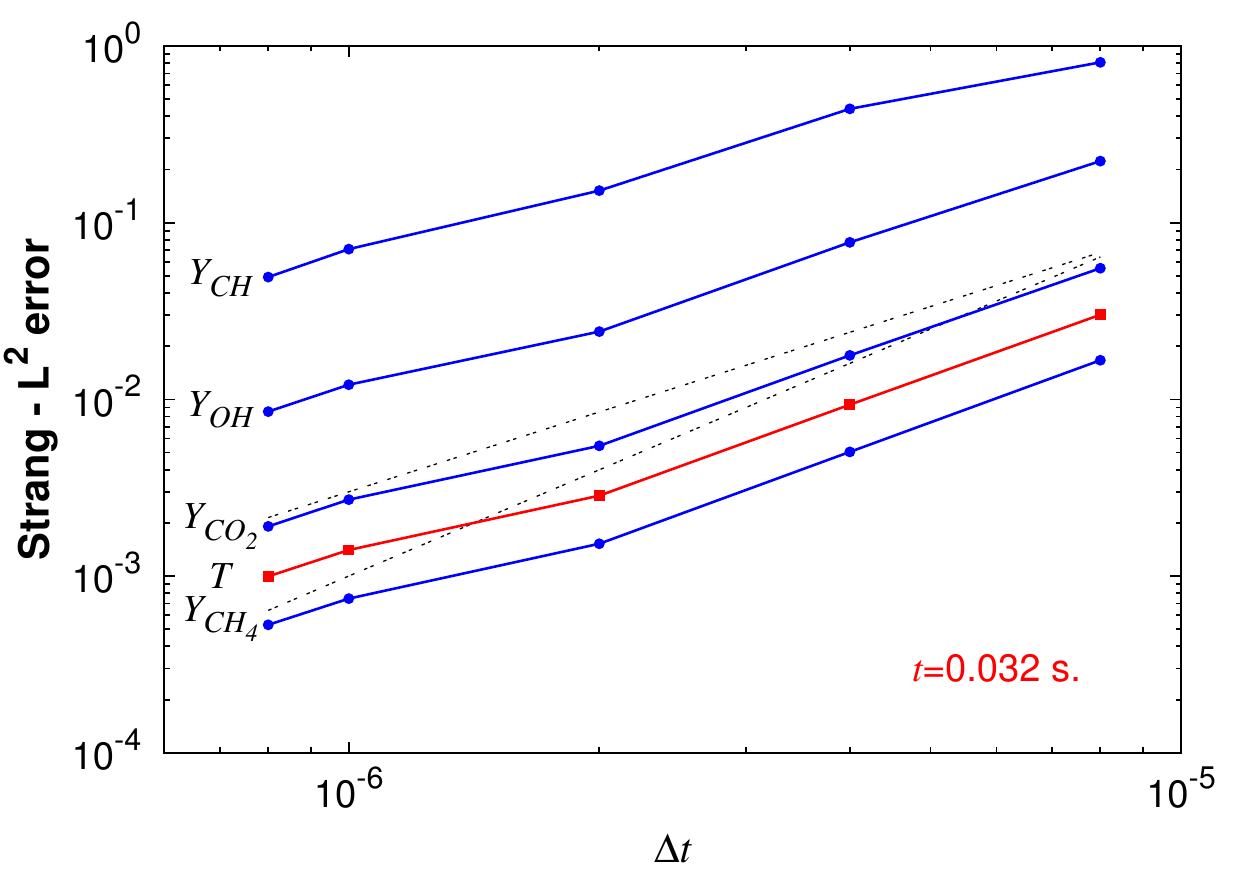}  
\caption{Local (left) and global (right) $L^2$ errors for the Strang
scheme for temperature $T$ and species $Y_{\rm CH_4}$, $Y_{\rm CO_2}$,
$Y_{\rm OH}$, and $Y_{\rm CH}$.
Lines with slope of 3 and 2.5 (left),
and of 2 and 1.5 (right) are depicted.}
\label{fig:str_norms_CC}
\end{figure}

Figures \ref{fig:lie_norms_CC}
and \ref{fig:str_norms_CC} illustrate the
local and global errors for the Lie and Strang
splitting approximations, respectively,
for relatively large splitting time steps.
We consider temperature $T$, and some main 
and minor species like $Y_{\rm CH_4}$ and $Y_{\rm CO_2}$,
and $Y_{\rm OH}$ and $Y_{\rm CH}$, respectively.
Global errors are evaluated at time $0.032\,$s, 
which corresponds to a maximum pulsated velocity.
For the Lie approximations, 
the dependence on the splitting time step $\split$ varies
from $\split^2$ to $\split^{1.5}$
(close to $\split$ for $Y_{\rm CH}$);
and similarly from 
$\split^3$ to $\split^{2.5}$,
for the Strang
solutions in Figure \ref{fig:str_norms_CC}
(about 
$\split^2$ 
for $Y_{\rm CH}$).
Global errors follow approximately the same
behaviors. However, compensations can take place
as illustrated, for instance, by the Strang scheme
that displays 
behaviors between $\split^2$ and $\split^{1.5}$,
even 
for very large splitting time steps.
The Lie scheme on the other side involves
a global 
accuracy that behaves like $\split^{0.5}$,
and even worse 
for very large time steps.
In what concerns to the present study
we can identify
similar behaviors previously observed for the KPP problem,
and predicted by the theoretical study,
this time for a much more complex problem.
In particular  
splitting approximations with relatively large time steps involve
better accuracies 
than those expected out of the asymptotic bounds.
Moreover, splitting errors remain bounded even for considerably large
time steps of about $10^{-5}\,$s, compared, for instance, with 
some of the chemical time scales, of the order of the nanoseconds.
Complementary analyses on these numerical results,
as well as more details and further extensions of this approach for Low Mach number
flames will be reported in a forthcoming work.

%% file: Annexe_CF.tex
\section{Laminar premixed counterflow flame}\label{Annexe_CF}
The counterflow, premixed methane flame configuration 
is illustrated in Figure \ref{fig:setup}.
Two premixed flames are stabilized between two 
injections of the same mixture of methane and air
with an axial velocity of $1.423\,$m/s.
These jets are pulsated with a frequency of $100\,$Hz 
and a modulation of $10\%$, in a synchronized way, so that the plane $z=0$ 
remains in the symmetry plane.
The distance between the injectors and this stagnation planes is $d=1.55\,$cm.
Figure \ref{fig:profT} illustrates the velocity pulsations on the fuel injection,
and the time variations of temperature profiles, 
as well as for $Y_{\rm CH_4}$ and $Y_{\rm OH}$.
\begin{figure}[!htb]
  \centering
  \includegraphics[width=0.41\textwidth]{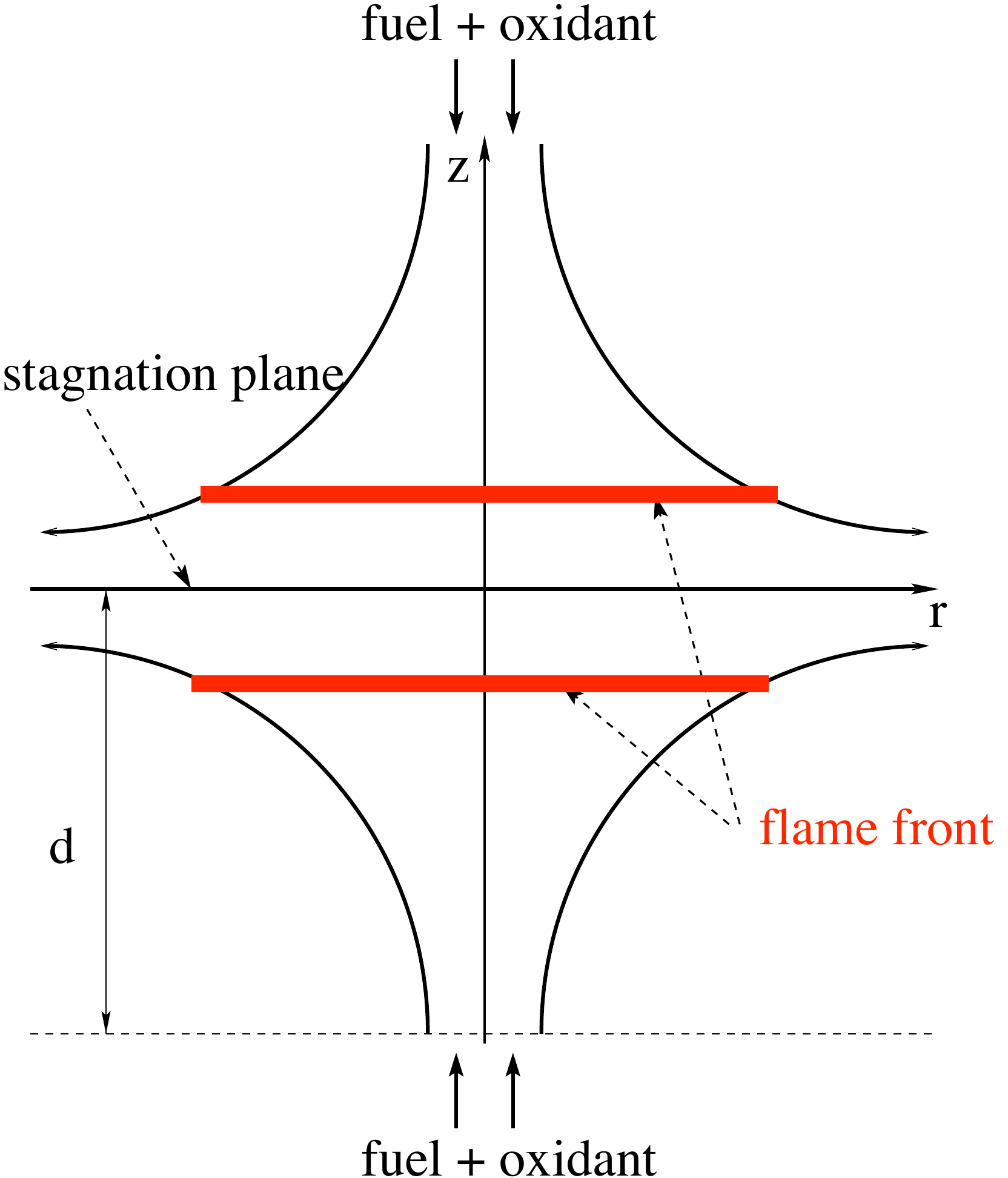} \\
  \caption{Schematic configuration of the counterflow, premixed flames.}
  \label{fig:setup}
\end{figure}
\begin{figure}[!htb]
  \centering
  \includegraphics[width=0.49\textwidth]{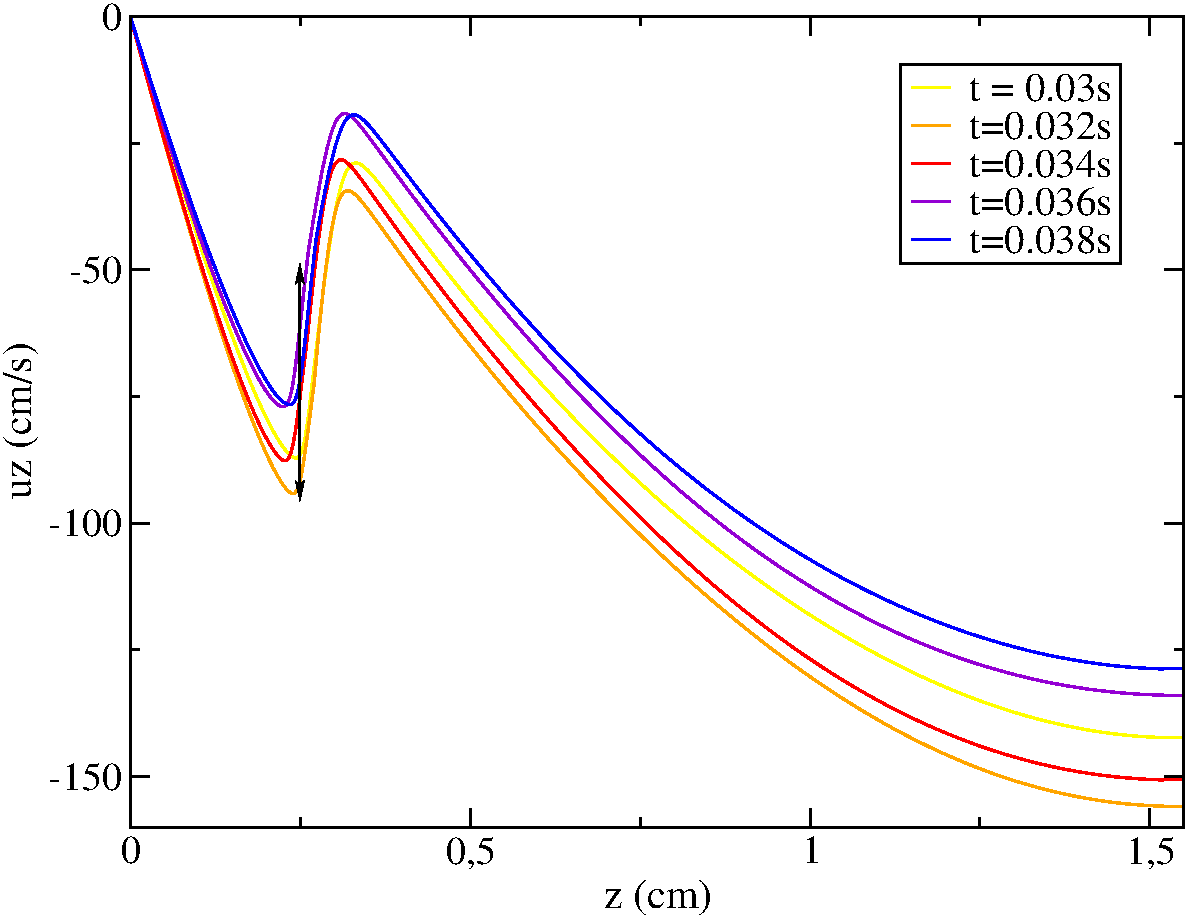} 
  \includegraphics[width=0.49\textwidth]{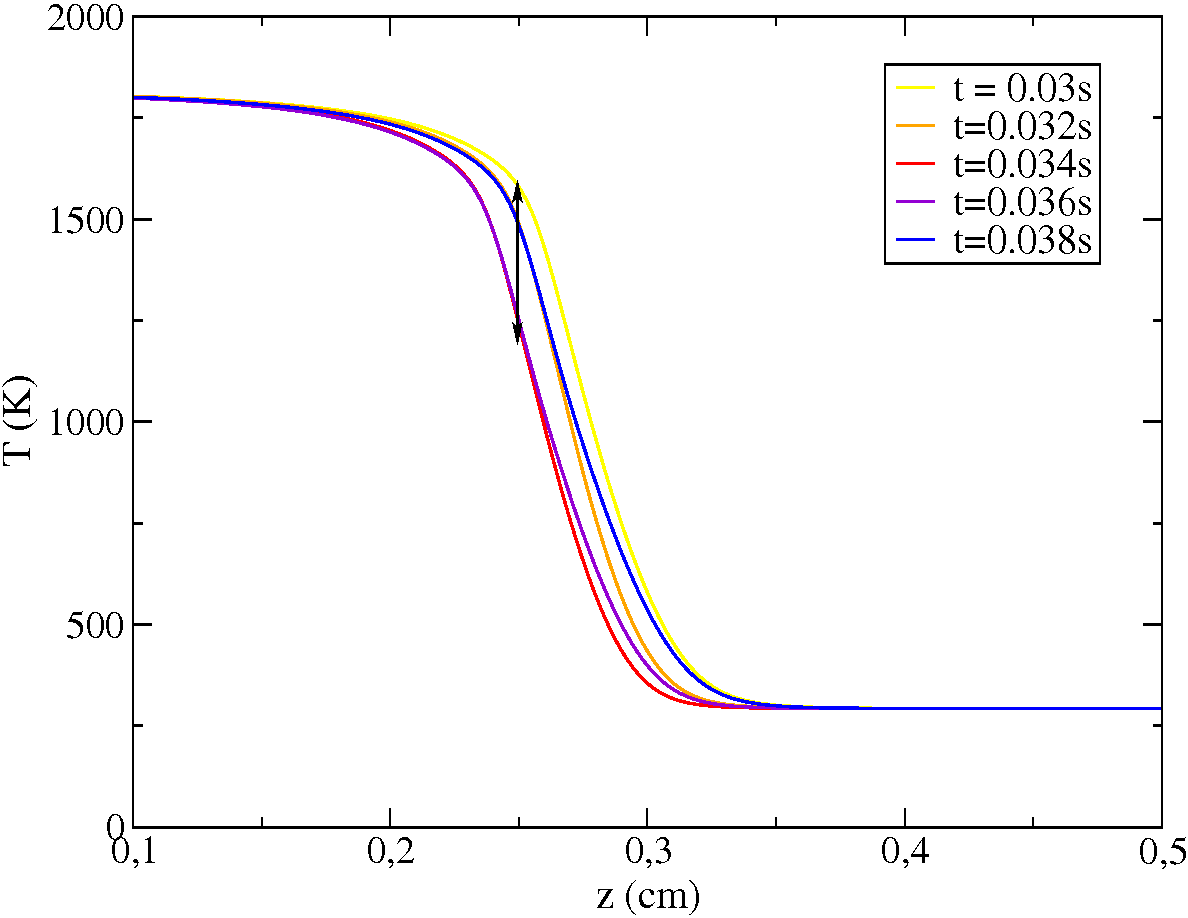} 
   \includegraphics[width=0.49\textwidth]{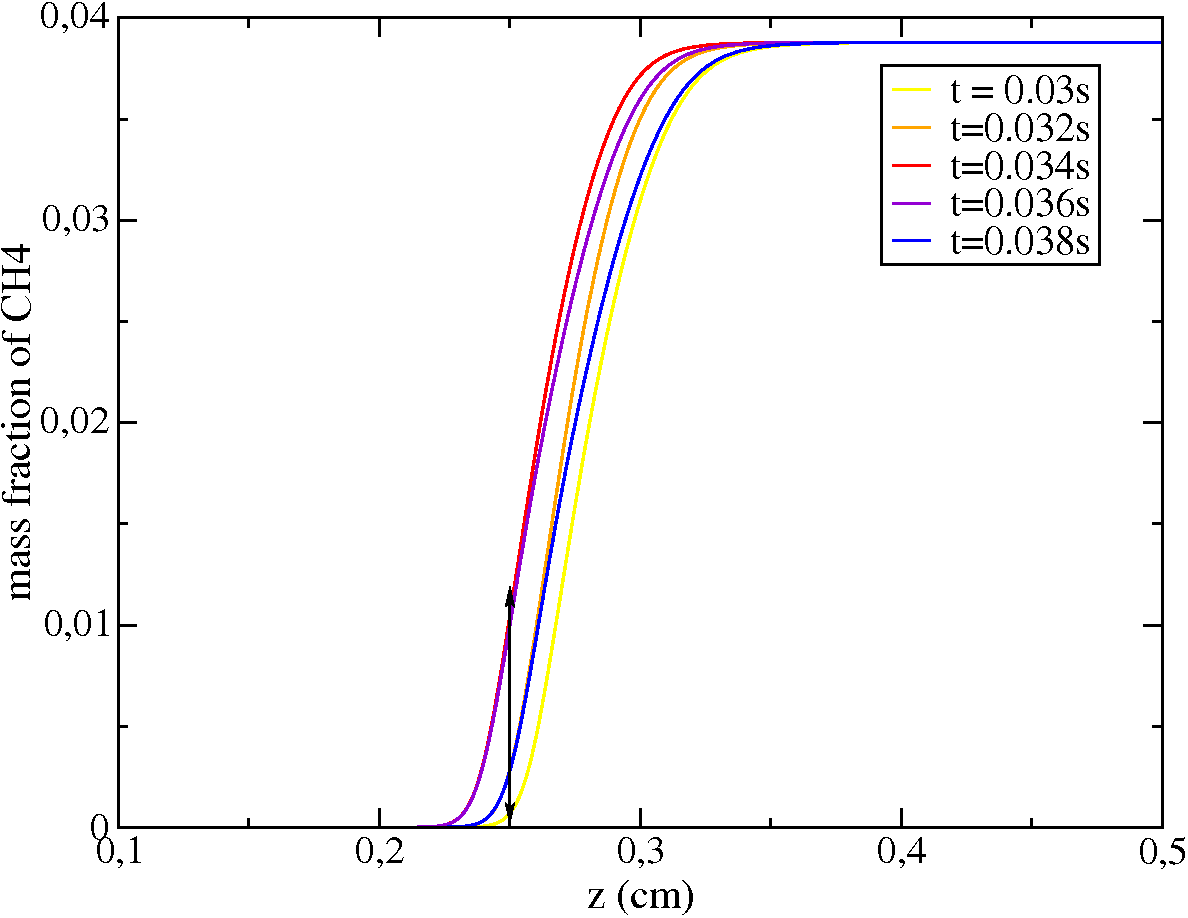} 
  \includegraphics[width=0.49\textwidth]{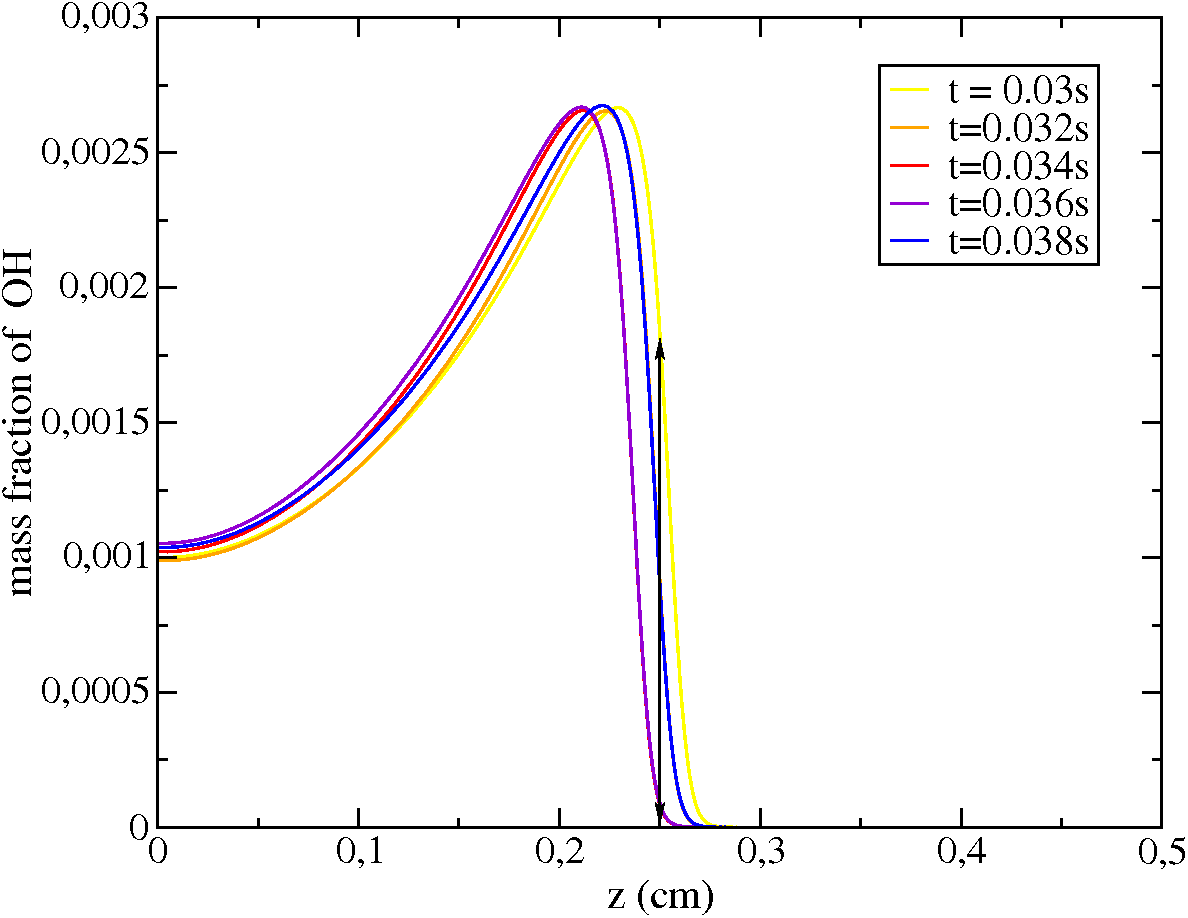} 
 \caption{Top: time variations of axial velocity $u_z$ (left) and temperature $T$ (right) profiles.
 Bottom: mass fraction profiles for $Y_{\rm CH_4}$ (left) and $Y_{\rm OH}$ (right).
 Position $z=0.25\,$cm is indicated.}\label{fig:profT}
\end{figure}